\documentclass[11pt]{amsart}
\usepackage[utf8]{inputenc}
\usepackage[T1]{fontenc}
\usepackage[frenchb,english]{babel} 
\usepackage{textcomp}
\usepackage{amsmath,amssymb}
\usepackage{amsthm}
\usepackage{lmodern}

\usepackage[a4paper,margin=3cm]{geometry}

\usepackage{graphicx}             
\usepackage{xcolor}               
\usepackage{microtype}          

\usepackage{hyperref}
\hypersetup{pdfstartview=XYZ}

\usepackage{bbm}
\usepackage{mathrsfs}
\usepackage{subfig}

\title{Rescaled bipartite planar maps converge to the Brownian map}
\author{C\'eline Abraham}
\date{July 2014}
\address{D\'epartement de math\'ematiques, Universit\'e Paris-Sud, 91405 ORSAY Cedex, FRANCE}
\email{celine.abraham@math.u-psud.fr}
\newcommand{\Z}{\mathbbm{Z}}
\newcommand{\N}{\mathbbm{N}}
\newcommand{\R}{\mathbbm{R}}

\def\build#1_#2^#3{\mathrel{
\mathop{\kern 0pt#1}\limits_{#2}^{#3}}}

\newtheorem{theorem}{Theorem}
\newtheorem{proposition}[theorem]{Proposition}

\newtheorem{lemma}[theorem]{Lemma}
\newtheorem{remark}[theorem]{Remark}

\begin{document}
\maketitle

\begin{abstract}
For every integer $n\geq 1$, we consider a random planar map $\mathcal{M}_n$
which is uniformly distributed over the class of all rooted 
bipartite planar maps with $n$ edges. We prove that the vertex set 
of $\mathcal{M}_n$ equipped with the graph distance rescaled by the factor $(2n)^{-1/4}$
converges in distribution, in the Gromov-Hausdorff sense, to the Brownian map.
This complements several recent results giving the convergence of various classes of random planar maps
to
the Brownian map.
\end{abstract}

\section{Introduction} 

Much attention has been given recently to the convergence of large random planar maps viewed as 
metric spaces to the continuous random metric space known as the Brownian map. See in
particular \cite{albenque,quadwithnopending,uniqueness,brownianmapmiermont}. 
The main goal of the present work is to provide another interesting example of these 
limit theorems, in the case of bipartite planar maps with a fixed number of edges. 

Recall that a planar map is a proper embedding of a finite connected graph in the two-dimensional sphere, viewed up to orientation-preserving homeomorphisms of the sphere. The faces of the map are the connected components of the complement of edges. The degree of a face is the number of edges incident to it, with the convention that, if both sides of an edge are incident to the same face, then this edge is counted twice in the degree of the face. 
A planar map is rooted if there is a distinguished oriented edge, which is called the root edge. 

We consider only bipartite planar maps in the present work. A planar map is bipartite if its vertices can be colored with two colors, in such a way that two vertices that have the same color are not connected by an edge (in particular,
there are no loops). This is equivalent to the property  that all faces of the map have an even degree. 

If $M$ is a planar map, the vertex set of $M$ is denoted by $V(M)$, and the usual graph
distance on $V(M)$ is denoted by $d_{\rm{gr}}^M$.
Let $\mathbf{M}^b_n$ stand for the set of all  rooted  bipartite maps with $n$ edges.

\begin{theorem}
\label{cvgcarte} 
For every $n\geq 1$,
let $\mathcal{M}_n$ be uniformly distributed over $\mathbf{M}^b_n$. Then,
$$\left(V(\mathcal{M}_n) , 2^{-1/4} n^{-1/4} d_{\rm{gr}}^{\mathcal{M}_n}\right)  \build{\longrightarrow}_{n \to \infty}^{(d)} (\mathbf{m}_{\infty}, D^*) $$
where $(\mathbf{m}_{\infty}, D^*)$ is the Brownian map. The convergence holds in distribution in the space  $(\mathbb{K}, d_{GH})$, 
where  $\mathbb{K}$ is the set of all isometry classes of compact metric spaces and $d_{GH}$ is the Gromov-Hausdorff distance. 
\end{theorem}

A brief presentation of the Brownian map will be given in Section \ref{CvgCarte} below. See \cite{uniqueness}
and the references therein for more information about this random compact metric space.

As mentioned above, several limit theorems analogous to Theorem \ref{cvgcarte} have been
proved for other classes of random planar maps. The case of $p$-angulations,
which are planar maps where all faces have the same degree $p$, has received particular
attention. Le Gall \cite{uniqueness} proved 
the convergence in distribution of rescaled $p$-angulations with a fixed number of faces to the
Brownian map, both when $p=3$ (triangulations) and when $p\geq 4$ is even. The case 
of quadrangulations ($p=4$) has been treated independently by Miermont \cite{brownianmapmiermont}.
More recently, similar results have been obtained for random planar maps 
with local constraints:
Beltran and Le Gall \cite{quadwithnopending} proved the convergence to the Brownian map for 
quadrangulations with no pendant vertices, and Addario-Berry and Albenque \cite{albenque} discussed similar results
for simple triangulations or quadrangulations, where there are no loops or multiple edges. 

All these papers however deal with random planar maps conditioned to have a fixed
number of faces. In our setting, it would make no sense to
consider the uniform distribution over all bipartite planar maps with a given number of faces,
since there are infinitely many such planar maps. Similarly it would make no sense 
to condition on the number of vertices, and for this reason we consider conditioning
on the number of edges, which results in certain additional technical difficulties. 

In order to prove Theorem \ref{cvgcarte}, we first establish a similar result for planar maps that are both rooted and pointed (this means that, in addition to the root edge there is a distinguished vertex, which we call the origin of the map). 

As in several of the previously mentioned papers, the proof of this result relies on the combinatorial bijections of Bouttier, di Francesco and Guitter \cite{BDG} between (rooted and pointed) bipartite planar maps and certain labeled two-type plane trees. Let $\mathbf{M}^{b \bullet}_n$ denote the set of all rooted and pointed planar bipartite maps with $n$ edges, and let $\mathcal{M}_n^{\bullet}$ be uniformly distributed over $\mathbf{M}^{b \bullet}_n$. The random
tree associated with $\mathcal{M}_n^{\bullet}$ via the Bouttier, di Francesco, Guitter bijection is identified as a (labeled) two-type Galton-Watson tree
with explicit offspring distributions, conditioned to have a fixed total progeny (see Proposition \ref{loiarbre}
below). In order to prove the convergence to the Brownian map, an important technical step
is then to derive asymptotics for the contour and label functions associated with this conditioned tree (Theorem
\ref{cvgcontourlabel}).
Such asymptotics for conditioned two-type Galton-Watson trees have been discussed in \cite{MM} and
\cite{Mtrees}. However both these papers consider conditioning on the number 
of vertices of one type, which makes it easier to derive the desired asymptotics from the case of
usual (one-type) Galton-Watson trees. The fact that we are here conditioning on the total number
of vertices creates a significant additional difficulty, which we handle through an 
absolute continuity argument similar to the ones used in Section 6 of \cite{Ito}. A useful 
technical ingredient is a seemingly new definition of a ``modified'' Lukasievicz path associated with a two-type tree,
which might be of independent interest. This new definition is somehow related
to a bijection of Janson and Stef\'ansson \cite{J} between one-type and two-type trees.

As we were finishing the first version of the present article, we learnt of the very recent paper \cite{BJM},
which obtains a result similar to ours for {\it general} planar maps. The arguments of \cite{BJM} might also
be applicable to the bipartite case, but the methods seem quite different from the ones that are
presented here. 

We finally note the simple scaling constant $2^{-1/4}$ in Theorem \ref{cvgcarte}. As far as we know, this value is different from the ones already computed for other classes of maps. The analogous constant for uniform general maps with $n$ edges \cite{BJM} is $\left(9/8\right)^{1/4}$ and the one for uniform quadrangulations with $n$ edges ($n$ even) \cite{uniqueness,brownianmapmiermont}  is $\left( 9/4 \right)^{1/4}$.

The paper is organized as follows. Section 2 introduces our main notation and definitions, and 
recalls the key bijection of \cite{BDG}  between rooted and pointed bipartite maps and labeled
two-type trees. In Section 3, we identify
the distribution of the random two-type tree associated to a map uniformly distributed over $\mathbf{M}_n^{b \bullet}$, 
and we introduce its ``modified'' Lukasievicz path. Section 4 is devoted to the asymptotics of the contour and label functions coding the two-type tree.  Section 5 gives the
proof of the statement analogous to Theorem \ref{cvgcarte} for rooted and pointed maps. Finally, Section 6 explains how to derive Theorem \ref{cvgcarte} from the latter statement.

\section{Bipartite planar maps and trees}

\subsection{Trees}
\label{sec:trees}
We set $\N=\lbrace 1, 2, \dots \rbrace$ and by convention $\N^0= \lbrace \emptyset \rbrace$. We introduce the set $$ \mathcal{U}= \bigcup_{n=0}^{\infty} \N^n.$$
An element of $\mathcal{U}$ is a sequence $u=(u^1, \dots, u^n)$ of elements of $\N$, and we set $|u|=n$ so that $|u|$ represents the ``generation'' of $u$.
If $u=(u^1, \dots, u^n)$ and $v=(v^1, \dots, v^m)$ are two elements of $\mathcal{U}$, then $uv=(u^1,\dots,u^n,v^1,\dots,v^m)$ is the concatenation of $u$ and $v$. 
The mapping $\pi : \mathcal{U} \setminus \lbrace \emptyset \rbrace \rightarrow \mathcal{U}$ is defined by $\pi((u^1, \dots, u^n))=(u^1, \dots u^{n-1})$. One says that $\pi(u)$ is the parent of $u$, or that $u$ is a child of $\pi(u)$. 
A plane tree $T$ is a finite subset of $\mathcal{U}$ such that
\begin{enumerate}
\item[(i)] $\emptyset \in T$;
\item[(ii)] if $u \in T \setminus \lbrace \emptyset \rbrace$, then $\pi(u) \in T$;
\item[(iii)] for every $u \in T$, there exists an integer $k_u(T) \geq 0$ such that, for every $j \in \N$, $uj \in T$ if and only if $1 \leq j \leq k_u(T)$. 
\end{enumerate}
In (iii), the number $k_u(T)$ is interpreted as the number of children of $u$ in $T$. 
The size of a plane tree $T$ is $|T|=\# T-1$, which is the number of edges of $T$. 
We denote  the set of all plane trees by $\mathbf{A}$. 

Consider now a plane tree $T$ and $n= |T|$. We introduce the contour  sequence $(u_0,u_1,\ldots,u_{2n})$
of $T$, which is defined by induction as follows :  $u_0=\emptyset$ and for $i \in \lbrace 0, \dots 2n-1 \rbrace$, $u_{i+1}$ is either the first child of $u_i$ that has not appeared yet in the sequence $(u_0,\dots,u_i)$, or the parent of $u_i$ if all the children of $u_i$ already appeared in the sequence $(u_0, \dots, u_i)$. Note that $u_{2n}=\emptyset$ and that all vertices of $T$
appear in the sequence $(u_0,\ldots,u_{2n})$ (some appear more than once). 

The white vertices of a tree $T$ are all vertices $u$ such that $|u|$ is even and similarly the black vertices are all vertices such that $|u|$ is odd.   
We denote  the sets of white and black vertices of $T$  by $T^{0}$ and $T^1$ respectively.  

We will be interested in certain two-type Galton-Watson trees, which we briefly describe here. Let $(\mu_0,\mu_1)$ be a pair of probability distributions on $\Z_+$ with respective (finite)  means $m_0$ and $m_1$. We only consider pairs such that $\mu_0(1)+\mu_1(1)<2$ and $m_0m_1 \neq 0$. We say that $(\mu_0,\mu_1)$ is subcritical if $m_0m_1<1$ and critical if $m_0m_1=1$. Assume that the pair $(\mu_0, \mu_1)$ is critical or subcritical.  
A random tree $\xi$ whose distribution is specified by 
$$ P(\xi=T)= \prod_{u \in T^0} \mu_0(k_{T}(u)) \prod_{u \in T^1} \mu_1(k_{T}(u))\ ,\quad\forall T\in \mathbf{A}$$
is called a two-type Galton-Watson tree with offspring distributions $(\mu_0, \mu_1)$. Informally, white vertices have children according to the offspring distribution $\mu_0$ and black vertices have children according to $\mu_1$. 

We now introduce labeled trees. A labeled tree is a pair $(T, (\ell(u))_{u \in T^0})$ where $T$ is a plane tree and $(\ell(u))_{u \in T^0}$ is a collection of labels assigned to the white vertices of $T$, which must satisfy the following properties.
\begin{enumerate}
\item[(i)] For every $u\in T$, $\ell(u)\in\Z$. 
\item[(ii)] Let $v \in T^1$ and $k=k_v(T)$. Let $v_1=v1,\ldots, v_k=vk$ be the children 
of $v$ in $T$, and set also $v_0=v_{k+1}=\pi(v)$. Then, for every $i\in\{0,1,2,\ldots,k\}$, $\ell(v_{i+1})\geq \ell(v_i)-1$. 
\end{enumerate}
The number $\ell(u)$ is called the label of $u$. Property (ii) means that, if $v$ is a black vertex,
the labels $i$ and $j$ of two white vertices adjacent to $v$ and consecutive in clockwise order around $v$ satisfy $j \geq i-1$. 

We denote  the set of all labeled trees with $n$ edges by $\mathbf{T}_n$.

\begin{figure}[h]
\centering
\includegraphics[height=5cm]{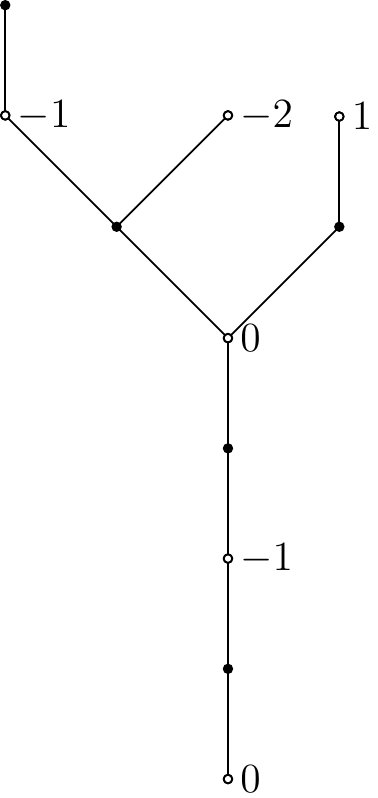}\qquad
\hbox{\includegraphics[width=5cm]{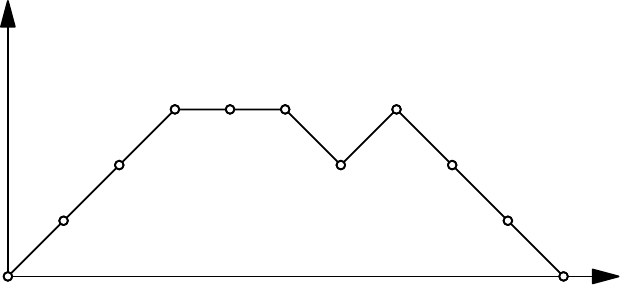}\vspace{1cm}}
\qquad
\includegraphics[width=5cm]{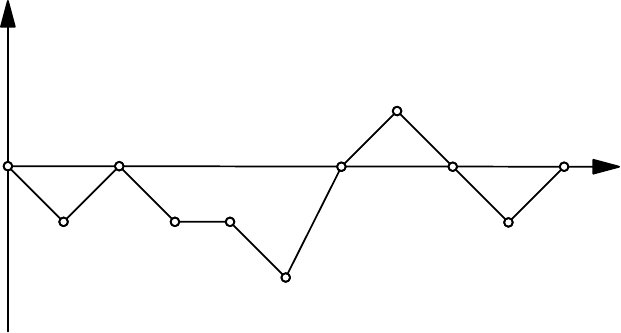}
\caption{A labeled tree $T$ with $n=10$ edges, the contour function $C^{T^0}$ and the label function $L^{T^0}$. }
\label{fig:twotypetree}
\end{figure}

A  labeled tree $(T, (\ell(u))_{u \in T^0})$  can be coded by a pair of functions. Recall that if $|T|=n$, $(u_0, \dots, u_{2n})$ is the contour sequence of $T$. Note that $u_i$ is white if $i$ is even and black if $i$ is odd. We define for $0 \leq i \leq 2n$,
$$C^{T}_i=|u_i|.$$
We extend $C^{T}$ to the real interval $[0,2n]$ by linear interpolation. The function $C^{T}$ is the contour function of the tree $T$. 
For  $0 \leq i \leq n$, set $v_i= u_{2i}$. The sequence $(v_0, \dots, v_n)$ is called the white contour sequence.  
We then set, for $0 \leq i \leq n$,
$$C^{T^0}_i=\frac{1}{2} |v_i|$$ and 
$$L^{T^0}_i=\ell(v_i).$$
We notice that  for $0 \leq i \leq n$ , we have $C^{T^0}_i= \frac{1}{2} C^{T}_{2i}$. 
We also extend both $C^{T^0}$ and $L^{T^0}$ to to the real interval $[0,n]$ by linear interpolation. 
The function $C^{T^0}$ is called the contour function of $T^0$ (or the white contour function) and $L^{T^0}$ is called the label function of $T^0$.  See Fig.1 for an example. It is easy to verify that the labeled tree $(T, (\ell(u))_{u \in T^0})$ is uniquely determined by the pair   $(C^{T}, L^{T^0})$ (on the other hand, the pair $(C^{T^0}, L^{T^0})$ does not give enough information to recover the tree).

\subsection{The Bouttier-Di Francesco-Guitter bijection}
\label{bijBDG}
In this section we describe the Bouttier-Di Francesco-Guitter bijection (BDG bijection) between $\mathbf{T}_n \times \lbrace 0,1\rbrace$ and $\mathbf{M}_n^{b \bullet}$. This construction can be found in \cite{BDG} and in \cite{uniqueness} in the particular case of $2p$-angulations. 
  
We start with a labeled tree $(T, (\ell(u))_{u \in T^0}) \in \mathbf{T}_n$ and $\epsilon \in \lbrace 0,1 \rbrace$.  
As above, $(v_0, \dots, v_n)$ stands for the white contour sequence of $T$. 
We suppose that the tree $T$ is represented in the plane 
in the (obvious) way as suggested by Fig.1. A corner of $T$ is a sector around a vertex of $T$ delimited by two consecutive edges in clockwise order. Each corner is given the label of its associated vertex. We note that every $i\in\{0,1,\ldots,n-1\}$
corresponds to exactly one corner of the vertex $v_i$ (if we move around the tree in clockwise order, the successive white
vertices that are visited are $v_0,v_1,\ldots,v_{n-1},v_n=v_0$ and each visit but the last one corresponds to a new corner), and we will abuse terminology by calling this corner the corner $v_i$.

We then add an extra vertex $\partial$ outside the tree $T$, and  
we construct a planar map $M^{\bullet}$, whose vertex set is the union of $T^0$ 
and of the extra vertex $\partial$, as follows: For every $i \in \lbrace 0, \dots, n-1 \rbrace$,
\begin{itemize}
\item if $\ell(v_i)= \min \lbrace \ell(v), v \in T^0 \rbrace$, then we draw an edge of $M^{\bullet}$ between the corner $v_i$ and $\partial$;
\item if $\ell(v_i) > \min \lbrace \ell(v), v \in T^0 \rbrace$, then we draw an edge of $M^{\bullet}$ between the corner $v_i$ and the corner $v_j$, where $j=\min\{k>i:\ell(v_k)=\ell(v_i)-1\}$ if $\{k>i:\ell(v_k)=\ell(v_i)-1\}$ is nonempty, $j=\min\{k\geq 0: \ell(v_k)=\ell(v_i)-1\}$ otherwise. 
\end{itemize}
Thanks to property (ii) of the labels, it is possible to achieve this construction in such a way that edges do not intersect (except at their ends) and do not cross the edges of the tree. 
The collection of all edges drawn in the preceding construction gives a bipartite planar map $M^{\bullet}$ with $n$ edges. We then declare that the vertex $\partial$ is the distinguished vertex of this map and that its root edge is the edge obtained at step $i=0$ of the preceding construction. 
 The parameter $\epsilon$ gives the orientation of this root edge: the root vertex is $\emptyset$ if and only if  $\epsilon=0$.  
In this way we get a pointed and rooted bipartite planar map $M^{\bullet}$. See Fig.2 for an example with $\epsilon=0$.
 
\begin{figure}[h]
\centering
\includegraphics{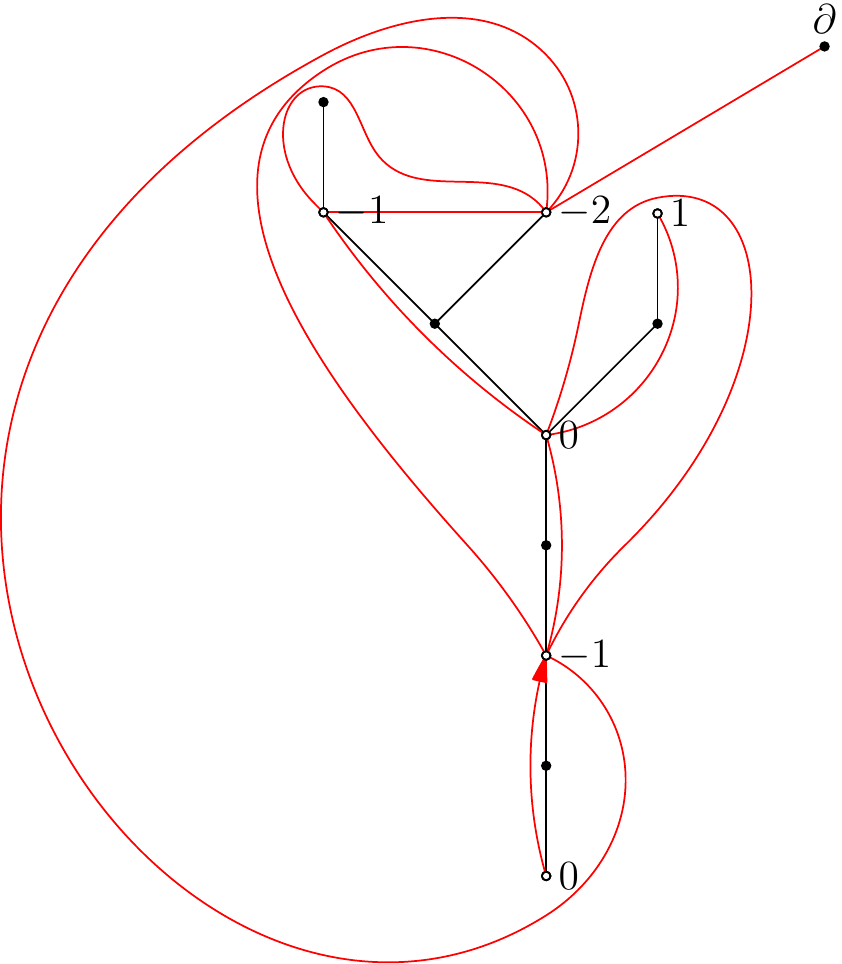}
\caption{The  labeled tree $T$ of Fig.1 and the associated rooted and pointed bipartite map $M^{\bullet}$. }
\label{fig:bdg}
\end{figure}

The preceding construction yields a bijection from $\mathbf{T}_n \times \lbrace 0,1 \rbrace$ onto $\mathbf{M}^{b \bullet}_n$, which is called the Bouttier-Di Francesco-Guitter (BDG) bijection. 
In this bijection, white vertices of the tree $T$ are identified with vertices of the map $M^{\bullet}$ other than $\partial$, and moreover graph distances (in $M^{\bullet}$) from $\partial$ are related to labels on $T$ by the formula
\begin{equation}
\label{lienlabeldistance} 
d_{\text{gr}}^{M^{\bullet}}(\partial, u)= \ell(u)- \min \lbrace \ell(v), v \in T^0 \rbrace +1,
\end{equation}
for every $u \in T^0$. 
There is no such expression for $d_{\text{gr}}^{M^{\bullet}}(u,v)$ when $u$ and $v$ are arbitrary vertices of $M^{\bullet}$, but
the following bound will be very useful. Let $i, j \in \lbrace 0, \dots, n\rbrace$ such that $i<j$. Then,
\begin{equation}
\label{majolabeldistance}
d_{\text{gr}}^{M^{\bullet}}(v_i, v_j) \leq \ell(v_i) + \ell(v_j)-2 \max \lbrace \min \lbrace \ell(v_k), i \leq k \leq j \rbrace, \min \lbrace \ell(v_k), j \leq k \leq i+n \rbrace \rbrace +2,
\end{equation} 
where we made the convention that $v_{n+k}=v_k$ for $0 \leq k \leq n$.
The proof of this bound is easily adapted from \cite[Lemma 3.1]{topostructure}. 

\section{Random trees and their contour functions}
\subsection{The tree associated with a map chosen uniformly in $\mathbf{M}^{b \bullet}_n$.}
\label{lois}
Let $\mathcal{M}_n^{\bullet}$ be uniformly distributed over the set $\mathbf{M}^{b \bullet}_n$, as in Section 1. We let $(\mathcal{T}_n,(\ell_n(u))_{u \in \mathcal{T}_n^0})$ be the random labeled tree  associated with $\mathcal{M}_n^{\bullet}$ by the  previously described BDG bijection. The next proposition determines the distribution of this random tree. 

\begin{proposition}
\label{loiarbre}
Let $(\mu_0,\mu_1)$ be the pair of probability measures on $\mathbb{Z}_+$ defined by
$$\left\lbrace \begin{aligned}
                \mu_0(k)&=\frac{2}{3} \left(\frac{1}{3} \right)^k \\
                \mu_1(k)&= \frac{3}{8} \binom{2k+1}{k} \left(\frac{3}{16} \right)^k  \\                 
             \end{aligned} \right.$$ 
for every integer $k \geq 0$. The mean of $\mu_0$ is $1/2$ and the mean of $\mu_1$ is $2$, so that
the pair $(\mu_0,\mu_1)$ is critical. 

Then the random tree $\mathcal{T}_n$ is a two-type Galton-Watson tree with offspring distributions $(\mu_0, \mu_1)$ conditioned to have $n$ edges. 
Furthermore, conditionally given $\mathcal{T}_n$, the labels $(\ell_n(u))_{u \in \mathcal{T}_n^0}$ are uniformly distributed over all admissible labelings.               
\end{proposition}

\begin{proof}
Clearly it is enough to determine the law of $\mathcal{T}_n$. We observe that, if $T$ is a plane tree and if $u$ is a black vertex of $T$ with $k$ children, there are $\binom{2k+1}{k}$ possible choices for the increments of labels of white vertices around $u$. Fix $a \in (0,1)$ and $b \in (0, 1/4)$, and set for every $k \geq 0$,
$$\left\lbrace \begin{aligned}
                \nu_0(k)&=(1-a) a^k \\
                \nu_1(k)&= B \binom{2k+1}{k} b^k  \\                 
             \end{aligned} \right.$$ 
where $B$ is determined by the requirement that $\nu_1$ is a probability measure on $\mathbb{Z}_+$:
 $$ B= \frac{2b \sqrt{1-4b}}{1-\sqrt{1-4b}}.$$  
Assume that $(\nu_0, \nu_1)$ is subcritical or critical. If $\theta$ is a two-type Galton-Watson tree with offspring distributions $(\nu_0, \nu_1)$, then, for every plane tree $T$ with $n$ edges,
$$P(\theta=T)= \prod_{u \in T^0} \nu_0(k_{T}(u)) \prod_{u \in T^1} \nu_1(k_{T}(u)).$$ 
Writing $N_0$, respectively $N_1$, for the number of white, respectively black, vertices of $T$, we get
$$\begin{aligned}
P(\theta=T) 
&= (1-a)^{N_0} a^{N_1} B^{N_1} b^{N_0-1} \prod_{u \in T^1}  \binom{2k_{T}(u)+1}{k_{T}(u)}   \\
&=\frac{1}{b} ((1-a)b)^{N_0} (aB)^{N_1} \prod_{u \in T^1}  \binom{2k_{T}(u)+1}{k_{T}(u)}. \\
\end{aligned}$$
On the other hand, the quantity $P(\mathcal{T}_n=T)$ is proportional to the number of possible labelings of $T$, so that
$$ P(\mathcal{T}_n=T)=c_n \prod_{u \in T^1}  \binom{2k_{T}(u)+1}{k_{T}(u)},$$
where $c_n$ is the appropriate normalizing constant. If $a$ and $b$ are such that
\begin{equation}
\label{ab}
(1-a)b=aB,
\end{equation}
noting that $N_0+N_1=n+1$, we see that $P(\mathcal{T}_n=T)$ coincides with
$P(\theta=T)$ up to a multiplicative constant that depends only on $n$, and it follows that 
\begin{equation}
\label{lawtreeconditioned}
P(\mathcal{T}_n=T)=P(\theta=T \,\big|\, |\theta|=n).
\end{equation}
The condition \eqref{ab} holds if 
$$ a= \frac{1}{3}, \  b= \frac{3}{16}.$$
Furthermore, for these values of $a$ and $b$, we can verify that the mean of $\nu_0$ is $1/2$ and the
mean of $\nu_1$ is $2$, so that the pair $(\nu_0,\nu_1)$ is critical. It then follows from the preceding
considerations and in particular from \eqref{lawtreeconditioned} that the law of $\mathcal{T}_n$ is as stated
in the proposition. \end{proof}



\begin{remark}
One can easily compute the respective variances $\sigma_0^2$ and $\sigma_1^2$ of the probability measures $\mu_0$ and $\mu_1$. For future reference, we record that 
$$\sigma_0^2=\frac{3}{4}\ , \ \sigma_1^2=\frac{15}{2}.$$ 
\end{remark}

\subsection{The white contour function and an associated random walk}

Consider a random labeled tree $(\mathcal{T}, (\ell(u))_{u\in \mathcal{T}})$, such that 
$\mathcal{T}$ is a two-type Galton-Watson tree
 with offspring distributions $(\mu_0,\mu_1)$ given by Proposition \ref{loiarbre}, and 
 conditionally on $\mathcal{T}$ the labels $(\ell(u))_{u\in \mathcal{T}}$
 are uniformly distributed among admissible labelings. Let $N$ denote the (random) number of edges of $\mathcal{T}$, and
write $(u_0, \dots, u_{2N})$ for the contour sequence of $\mathcal{T}$.

For every integer $k \geq 0$, we let the  $\sigma$-field $\mathcal{F}_k$ be
generated by the following random variables:
\begin{enumerate}
\item[$\bullet$] the quantity $k\wedge N$ and
the vertices $u_0,u_1, \dots, u_{2(k \wedge N)}$ of $\mathcal{T}$;
\item[$\bullet$] the labels $\ell(u_0),\ell(u_2),\ldots,\ell(u_{2(k\wedge N)})$ of the white vertices $u_0,u_2, \ldots, u_{2(k \wedge N)} $;
\item[$\bullet$] for every odd integer $i$  such that $0<i<2(k \wedge N)$, the quantity $k_{\mathcal{T}}(u_i)$ 
and the labels $\ell(u_ij)$, $1\leq j\leq k_{\mathcal{T}}(u_i)$, of the (white) children of the black vertex $u_i$.
\end{enumerate}

Fig.3 below gives a realization of the tree $\mathcal{T}$ and Fig.4 shows the information discovered by the $\sigma$-field $\mathcal{F}_k$ for $k=0,1,\ldots,5$. This information
should also include the labels of the white vertices that are successively revealed, but these labels are not shown here. 

\medskip
\begin{figure}[h]
\centering
\includegraphics{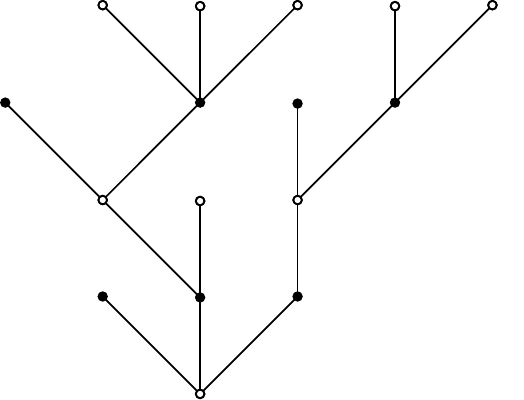}
\caption{A realization of the tree $\mathcal{T}$. }
\label{fig:tribuY-1}
\end{figure}

\begin{figure}[h]
\includegraphics{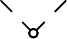}
\hspace{2mm}
\includegraphics{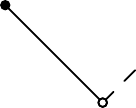}
\hspace{2mm}
\includegraphics{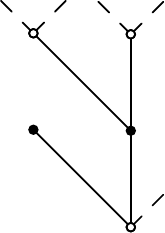}
\hspace{2mm}
\includegraphics{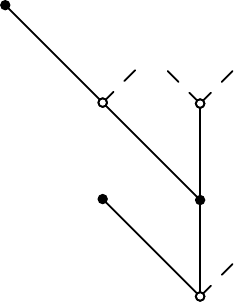}
\hspace{2mm}
\includegraphics{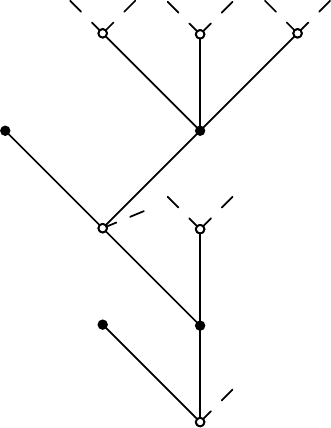}
\hspace{2mm}
\includegraphics{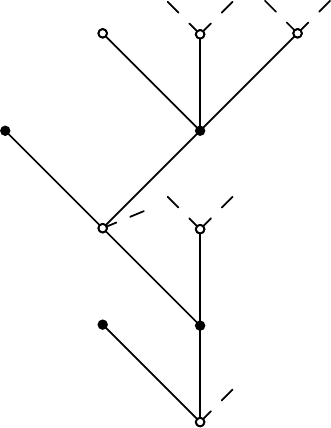}

\medskip
$\mathcal{F}_0\hspace{15mm} \mathcal{F}_1\hspace{18mm} \mathcal{F}_2\hspace{23mm} \mathcal{F}_3\hspace{24mm} \mathcal{F}_4\hspace{32mm} \mathcal{F}_5\hspace{15mm} $
\caption{The information about the tree $\mathcal{T}$ of Fig.3 given by the $\sigma$-field $\mathcal{F}_k$ for $k=0,1,\ldots,5$. The  dashed lines correspond to the ``active''
white vertices. In this example, 
$Y_0=Y_1=1$, $Y_2=Y_3=3$, $Y_4=6$, $Y_5=5$, etc. }
\label{fig:tribuY-2}
\end{figure}

We also introduce a random sequence $(Y_0,Y_1,\ldots Y_{N+1})$, which is defined by induction by setting $Y_0=1$ and, for every $0 \leq k \leq N$: 
\begin{itemize}
\item if $u_{2k}$ has at least one child that does not appear among $u_0, u_1, \dots, u_{2k-1}$, then $Y_{k+1}=Y_k+ k_{\mathcal{T}}(u_{2k+1})$,
\item otherwise $Y_{k+1}=Y_k-1$. 
\end{itemize}
Informally, for $0\leq k\leq N$, $Y_k$ counts the number of white vertices that have been visited before time $2k$
by the contour sequence, or
are children of black vertices visited before time $2k$, and are still ``active'' at time $2k$. Saying that a white vertex is still active means that it may have children that have not yet been visited at time $2k$. It is easy to verify that 
the random variable $Y_k$ (which is only defined on the $\mathcal{F}_k$-measurable set $\{k\leq N+1\}$) is $\mathcal{F}_k$-measurable and $Y_k \geq 1$ if $0 \leq k \leq N$, whereas $Y_{N+1}=0$. 

The white contour function of $\mathcal{T}$ can be expressed in terms of the sequence $(Y_k)_{0 \leq k \leq N+1}$ via the formula: for $0 \leq k \leq N$,
\begin{equation}
\label{lienCY}
C^{\mathcal{T}^0}_{k}= \text{Card} \lbrace j \in \lbrace 0, \dots, k-1 \rbrace : Y_j<\inf \lbrace Y_l: j+1 \leqslant l \leqslant k \rbrace \rbrace.
\end{equation}  
We leave the easy verification of \eqref{lienCY} to the reader.  Note that the sequence $(Y_0,Y_1,\ldots,Y_{N+1})$ is a 
kind of  ``Lukasiewicz path'' for our two-type tree, and that the preceding display is analogous to the 
formula relating the Lukasiewicz path of a (one-type) tree to its height function, see e.g. \cite[Proposition 1.2]{randomtrees}.
We also notice that the indices $j$ counted in $C^{\mathcal{T}^0}_{k}$ correspond to white vertices on the lineage path of $v_k$ in the tree $\mathcal{T}$.

For every $k \geq 0$, we denote the indicator function of the event 
\begin{center}
$\{k \leq N$ and the vertex $u_{2k}$ still has a non visited black child at instant $2k \}$
\end{center} 
by $\eta_k$. 
Then, conditionally on $\mathcal{F}_k$ and on the event $\{k \leq N \}$, $\eta_k$ is distributed as a Bernoulli random variable with parameter $\frac{1}{3}$. Furthermore, conditionally on $\mathcal{F}_k$ and on the event $\{ \eta_k=1 \}$, $k_{\mathcal{T}}(u_{2k+1})=Y_{k+1}-Y_k$ is distributed according to $\mu_1$. On the other hand, if $k\leq N$ and 
$\eta_k=0$, we have $Y_{k+1}=Y_k-1$.

Let $\nu$ be the probability measure on $\{-1, 0, 1, \dots \}$ defined by 
$$\left\lbrace
\begin{aligned}
&\nu(-1)= \frac{2}{3}\\
&\nu(k)= \frac{1}{3}\mu_1(k) \ \text{for} \ k \geq 0 \\
\end{aligned}
\right.$$
and let $(S_k)_{k \geq 0}$ be a random walk with jump distribution $\nu$ starting from $S_0=1$. It follows from the preceding discussion that $(Y_0, Y_1, \dots, Y_{N+1})$ has the same distribution as $(S_0, S_1, \dots, S_\tau)$ where $\tau= \inf \{n \geq 0, S_n=0 \}$. 
The distribution $\nu$ is centered and has a finite variance $\sigma^2=9/2$.

\begin{remark} 
It follows that $N+1$ (which is the total progeny of the two-type tree $\mathcal{T}$) has the same
distribution as $\tau$, and it is well known that this distribution is the same as the total progeny of 
a (one-type) Galton-Watson tree with offspring distribution $\mu(k)=\nu(k-1)$ for every $k\geq0$. 
A similar fact would hold for any (critical or subcritical) two-type Galton-Watson tree
such that the offspring distribution of white vertices is geometric. This 
was already observed in the recent   article of Janson and Stef\'ansson \cite{J}, with
a different approach involving a bijection between one-type and two-type trees: See \cite[Proposition 3.6]{CK}
for a statement derived from \cite{J}, which corresponds exactly to the previous discussion. 
\end{remark}

In the remaining part of this section, we state a couple of useful facts about the random walk $S$, which are
variants of results than can be found in \cite[Lemmas 1.9 to 1.12]{randomtrees}. For $m \in \Z_+$, we introduce the ``time-reversed'' random walk $ \hat{S}^m$ defined by
$$\hat{S}^m_k=S_m-S_{m-k}+1$$ for $0 \leq k \leq m$. 
The random walk $(\hat{S}^m_k, 0 \leq k \leq m)$ has the same distribution as $(S_k, 0 \leq k \leq m)$.
We set 
$$ M_m= \sup \lbrace S_k, 0 \leq k \leq m \rbrace $$ and $$ I_m= \inf \lbrace S_k, 0 \leq k \leq m \rbrace. $$
For every sequence  $\omega=(\omega(0), \omega(1), \dots)$ of integers of length at least $m$, we set  
$$F_m(\omega)=   \text{Card} \lbrace k \in \lbrace 1, \dots m \rbrace: \omega(k) >\sup \lbrace \omega(j), 0 \leq j \leq k-1 \rbrace \rbrace. $$ 
We then define $(R_{m})_{m \geq 0}$  and $(K_{m})_{m \geq 0}$  by $$R_m= F_m(\hat{S}^m)\ ,\quad K_{m}= F_m(S).$$

Note that we have 
\begin{equation}
\label{lienRS}
R_{m}= \text{Card} \lbrace j \in \lbrace 0, \dots, m-1 \rbrace: \ S_j<\inf \lbrace S_l : j+1 \leq l \leq m \rbrace \rbrace.
\end{equation}
(compare with \eqref{lienCY}). 

\begin{lemma}
\label{1.9randomtrees}
We define by induction $T_0=0$, and for every integer $j\geq 1$,  
$$T_{j+1}= \inf \lbrace k>T_j: S_k>S_{T_j} \rbrace.$$
Then the random variables $(S_{T_j}-S_{T_{j-1}})_{j \geq 1}$ are independent and identically distributed, and the distribution of $S_{T_1}-S_{T_0}=S_{T_1}-1$ is given by $$P(S_{T_1}-1=k)=\frac{3}{2} \nu([k, \infty))$$ for $k \geq 1 $. 
\end{lemma}

\begin{proof}
The fact that the random variables $(S_{T_j}-S_{T_{j-1}})_{j \geq 1}$ are i.i.d. is immediate from the strong Markov property. 
Let $S'$ be a random walk with jump distribution $\nu$, starting from $S'_0=0$, and $T'_1=\inf \{k >0, S'_k \geq 0 \}$.
By \cite[Lemma 1.9]{randomtrees}, we have $P(S'_{T'_1}=k)= \nu([k, \infty))$ for every $k \geq 0$. Next it is clear that the law of $S_{T_1}-1$ coincides with the conditional law of $S'_{T'_1}$ knowing that $\{S'_{T'_1} >0 \}$. The desired result easily follows. 
\end{proof}

It follows that the distribution of $S_{T_1}-1$ has a finite first moment, given by $E(S_{T_1}-1)=3\sigma^2/4$. A simple argument using  the law of large numbers then shows that 
$$\frac{M_m}{K_{m}} \underset{m \to \infty} {\longrightarrow} \frac{3 \sigma^2}{4}$$
almost surely. The next lemma provides estimates for ``moderate deviations'' in this convergence.

\begin{lemma}
\label{1.11randomtrees}
Let $\epsilon \in \left(0, \frac{1}{4} \right)$. We can find  $\epsilon'>0$ and  an integer $n_0 \geq 1$ such that for $m \geq n_0$ et $ l \in \lbrace 0, \dots, m \rbrace$, we have the bound $$P\left( \left|M_l- \frac{3 \sigma^2}{4} K_{l} \right|> m^{1/4+\epsilon} \right) < \exp\left( -m^{\epsilon'} \right).$$
\end{lemma}

\begin{proof}
The arguments are easily adapted from the proof of Lemma 1.11 in \cite{randomtrees}.
\end{proof}
 
\section{Convergence of the contour and the label functions}
\label{cvgarbre}
We keep the notation
$(\mathcal{T}, (\ell(u))_{u\in \mathcal{T}})$ for a random labeled tree  such that 
$\mathcal{T}$ is a two-type Galton-Watson tree
 with offspring distributions $(\mu_0,\mu_1)$ given by Proposition \ref{loiarbre}, and 
 conditionally on $\mathcal{T}$ the labels $(\ell(u))_{u\in \mathcal{T}}$
 are uniformly distributed among admissible labelings. As previously, $N=|\mathcal{T}|$. 
In this section, we discuss the convergence as $n\to\infty$ of the conditional distribution of the pair  $(n^{-1/2}C^{\mathcal{T}^0}_{nt}, n^{-1/4} L^{\mathcal{T}^0}_{nt})_{0 \leq t \leq 1}$ knowing that $N=n$ (recall the notation $C^{\mathcal{T}^0}$
and $L^{\mathcal{T}^0}$ for the contour function and the label function of $\mathcal{T}^0$, see the
end of subsection \ref{sec:trees}). The whole section is devoted to the proof of the next theorem. 

\begin{theorem}
\label{cvgcontourlabel}
The conditional distribution of $$\left( \frac{1}{\sqrt{n}} C^{\mathcal{T}^0}_{nt}, \frac{1}{n^{1/4}} L^{\mathcal{T}^0}_{nt} \right)_{0 \leq t \leq 1}$$ knowing that $N=n$ converges as $n \to \infty$ to the law of 
$$ \left( \frac{4 \sqrt{2}}{9}\, \mathbf{e}_t, 2^{1/4} \,Z_t  \right)_{0 \leq t \leq 1}$$
where $\mathbf{e}$ is a normalized Brownian excursion and $Z$ is the Brownian snake driven by this excursion. 
\end{theorem}

\begin{remark}
\label{rqCt}
We note that, for every $i \in \lbrace 0, \dots N+1 \rbrace$, $C^{\mathcal{T}}_{2i}= 2 C^{\mathcal{T}^0}_i$ and $|C^{\mathcal{T}}_{2i+1}-C^{\mathcal{T}}_{2i}|=1$. From this trivial observation, the convergence in distribution of  Theorem \ref{cvgcontourlabel} also implies that $ \left( \frac{1}{2\sqrt{n}} C^{\mathcal{T}}_{2nt} \right)_{0 \leq t \leq 1} $ converges to $\left(\frac{4\sqrt{2}}{9} \mathbf{e}_t \right)_{0 \leq t \leq 1}$, and the latter convergence holds jointly with that of Theorem \ref{cvgcontourlabel}. This simple remark will be useful later. 
\end{remark}
We recall that a normalized Brownian excursion $\mathbf{e}$ is just a Brownian excursion conditioned to have 
duration $1$, and that the distribution of $Z$ can be described by saying that, conditionally on $\mathbf{e}$,
$(Z_t)_{0\leq t\leq 1}$ is a centered Gaussian process with continuous sample paths, with covariance
$$E[Z_sZ_t\mid \mathbf{e}]= \min_{s\wedge t\leq r\leq s\vee t} \mathbf{e}_r.$$
It will sometimes be convenient to make the convention that $\mathbf{e}_t=Z_t=0$ for $t>1$.
Later we will consider the Brownian snake driven by other types of Brownian excursion, or by reflected 
linear Brownian motion. Obviously this is defined by the same conditional distribution as above. 

As we already mentioned in the introduction, Theorem \ref{cvgcontourlabel} is closely related to
analogous statements proved in \cite{MM,Mtrees} for multitype Galton-Watson trees. A major difference
however is the fact that \cite{MM,Mtrees} condition on the number of vertices of one particular type,
and not on the total number of vertices in the tree. Apparently the latter conditioning (on the total size of the
tree) cannot be handled easily by the methods of \cite{MM,Mtrees}. See in particular the remarks
in \cite[p.1682]{MM}.

Let us turn to the proof. We will rely on formula \eqref{lienCY} for $C^{\mathcal{T}^0}$. In connection with this formula, we recall
that  
$(Y_0, Y_1, \dots, Y_{N+1})$ has the same distribution as $(S_0, S_1, \dots, S_\tau)$, where $(S_k)_{k \geq 0}$ is a random walk with jump distribution $\nu$ starting from $1$, and $\tau= \inf \{n \geq 0: S_n=0 \}$. It will be convenient to use the notation $P_j$ for a probability measure under which the random walk $S$ starts from $j$.
By standard local limit theorems (see e.g. Theorems 2.3.9 and 2.3.10
in \cite{LL}), we have 
\begin{equation}
\label{thlimlocfeller}
\lim_{m \to \infty} \sup_{j \in \Z}  \;\left(1\vee \frac{|j|^2}{m}\right)\,\left|\sqrt{m} P_j(S_m=0)- \frac{1}{\sigma \sqrt{2\pi}} \exp \left(-\frac{j^2}{2 \sigma^2 m} \right) \right| =0.
\end{equation}
Here $\sigma^2=9/2$ is the variance of the distribution $\nu$. 
We also recall Kemperman's formula (see e.g. \cite[p.122]{pitman}). Let $m\geq j\geq 1$ be two integers. Then,
\begin{equation}
\label{Kemperman}
P_j(\tau=m)= \frac{j}{m} P_j(S_m=0).
\end{equation}
Since $N+1$ has the same distribution as $\tau$ under $P_1$, by combining Kemperman's formula with
\eqref{thlimlocfeller}, we immediately get
\begin{equation}
\label{estimationpopu}
n^{3/2}P(N=n) \underset {n \to \infty} {\longrightarrow} \frac{1}{\sigma\sqrt{2\pi}}
 \ \text{ and } \ 
 n^{1/2} P(N \geq n) \underset {n \to \infty} {\longrightarrow} \frac{2}{\sigma \sqrt{2 \pi}}
\end{equation}

\subsection*{First step}
Let $\delta \in (0,1)$ and let $\Psi$ be a bounded continuous function on the space $\mathcal{C}([0,1],\R^2)$
of all continuous functions from $[0,1]$ into $\R^2$. Recall the definition of the 
$\sigma$-fields $\mathcal{F}_k$. We have
\begin{equation}
\label{expressiondebut}
\begin{aligned}
&E\left[\Psi \left( \left( \frac{1}{\sqrt{n}} C^{\mathcal{T}^0}_{nt}, \frac{1}{n^{1/4}} L^{\mathcal{T}^0}_{nt} \right), 0 \leqslant t \leqslant 1-\delta \right) \mathbf{1}_{\lbrace N=n \rbrace} \right]\\
&=E\left[\Psi \left( \left( \frac{1}{\sqrt{n}} C^{\mathcal{T}^0}_{nt}, \frac{1}{n^{1/4}} L^{\mathcal{T}^0}_{nt} \right), 0 \leqslant t \leqslant 1-\delta \right) \mathbf{1}_{\lbrace N  \geq \lceil (1-\delta)n \rceil \rbrace)} P(N=n \,| \,\mathcal{F}_{\lceil (1-\delta)n \rceil}) \right]. \\
\end{aligned} 
\end{equation}
We then need to study the term $P(N=n\, |\, \mathcal{F}_{\lceil (1-\delta)n \rceil})$.

We notice that, conditionally on $\{ N \geq \lceil (1-\delta) n \rceil\}$ and on the $\sigma$-field $\mathcal{F}_{\lceil (1-\delta)n \rceil}$, the sequence $(Y_{\lceil (1-\delta)n \rceil}, Y_{\lceil (1-\delta)n \rceil+1}, \dots, Y_{N+1})$ has the same distribution as a random walk with jump distribution $\nu$ starting from $Y_{\lceil (1-\delta)n \rceil}$ and stopped when it hits $0$. Thus,  we apply Kemperman's formula \eqref{Kemperman},  and we obtain, still on the event $\{ N \geq \lceil (1-\delta) n \rceil\}$,
\begin{equation}
\label{phi}
P(N=n\,| \,\mathcal{F}_{\lceil (1-\delta)n \rceil})= P_{Y_{\lceil (1-\delta)n \rceil}}(\tau=n+1-\lceil (1-\delta) n \rceil)=\Phi_n(Y_{\lceil (1-\delta)n \rceil})
\end{equation} 
 where $\Phi_n(j)=\frac{j}{m_n} P_j(S_{m_n}=0)$, for $0 \leq j \leq m_n$, and $m_n=n+1-\lceil (1-\delta) n \rceil=\lfloor \delta n \rfloor +1$. 
\begin{lemma}
\label{cvgphiY}
 We have $$\lim_{n \to \infty} \sqrt{n}\, E\! \left[ \mathbf{1}_{\{N \geq \lceil (1-\delta)n \rceil\}} \left|   n \Phi_n(Y_{\lceil (1-\delta)n \rceil})-f_\delta\!\left(\frac{Y_{\lceil (1-\delta)n \rceil}}{\sqrt{m_n}} \right) \right|  \right]=0,$$
where for every $x \geq 0$, 
$$f_{\delta}(x)= \frac{x}{\delta \sigma \sqrt{2 \pi}}\exp \left(-\frac{x^2}{2 \sigma^2} \right).$$
\end{lemma}

\begin{proof}
We use the  local limit theorem
\eqref{thlimlocfeller} to evaluate $ n \Phi_n(j)$. Remark that $$n \Phi_n(j)= \frac{n}{m_n} j P_j(S_{m_n}=0)$$ 
and  $ n/m_n \longrightarrow 1/\delta$ as $n\to\infty$. 
It
easily follows from \eqref{thlimlocfeller} that
$$ \lim_{n \to \infty} \sup_{0\leq j\leq m_n} \left| j P_j(S_{m_n}=0)- \frac{1}{\sigma \sqrt{2 \pi}} \frac{j}{\sqrt{m_n}} \exp \left(-\frac{j^2}{2\sigma^2 m_n} \right) \right|=0.$$
Thus we have 
$$\lim_{n \to \infty} \sup_{0\leq j\leq m_n} \left|n \Phi_n(j)-\frac{1}{\sigma \delta \sqrt{2 \pi}} \frac{j}{\sqrt{m_n}} \exp \left(-\frac{j^2}{2\sigma^2 m_n} \right) \right|=0. $$
Recalling the definition of $f_{\delta}$, we have thus obtained
\begin{equation}
\label{cvgphi}
\lim_{n \to \infty} \sup_{0 \leq j \leq m_n} \left| n \Phi_n(j)- f_\delta\!\left(\frac{j}{\sqrt{m_n}} \right) \right| =0,
\end{equation} 
and the result of the lemma follows using also \eqref{estimationpopu}. 
\end{proof}

The next step is given by the following lemma. 
\begin{lemma}
\label{cvgfCY}
We have 
$$\sqrt{n} \,E\!\left[\mathbf{1}_{\{N \geq \lceil (1-\delta)n \rceil\}} \left|f_\delta\!\left(\frac{Y_{\lceil (1-\delta)n \rceil}}{\sqrt{m_n}} \right)-f_\delta\!\left( \frac{3 \sigma^2}{4}\frac{C^{\mathcal{T}^0}_{\lceil (1-\delta)n \rceil}}{\sqrt{m_n}} \right)  \right| \right]\underset { n \to \infty} {\longrightarrow} 0.$$
\end{lemma}
\begin{proof}
From the fact that $(Y_0, \dots Y_{N+1})$  has the same distribution 
as $(S_0, \dots, S_\tau)$ under $P_1$, and formula \eqref{lienCY}, we get that the distribution of  $(Y_{\lceil (1-\delta)n \rceil}, C^{\mathcal{T}^0}_{\lceil (1-\delta)n \rceil},N)$ conditionally on $\{N \geq \lceil (1-\delta)n \rceil \}$ is the same as the distribution of $(S_{\lceil (1-\delta)n \rceil}, R_{\lceil (1-\delta)n \rceil},\tau-1)$ 
under $P_1$ conditionally on $\{\tau > \lceil (1-\delta)n \rceil \}$. 
Thus the left-hand side of \eqref{cvgfCY} can be written as
$$\sqrt{n}\, E_1\!\left[\mathbf{1}_{\{\tau  > \lceil (1-\delta)n \rceil\}} \left|f_\delta\!\left(\frac{S_{\lceil (1-\delta)n \rceil}}{\sqrt{m_n}} \right)-f_\delta\!\left( \frac{3 \sigma^2}{4}\frac{R_{\lceil (1-\delta)n \rceil}}{\sqrt{m_n}} \right)  \right| \right].$$
By time reversal, the following identity in distribution holds
under $P_1$, for $0 \leq l \leq m$ : $$(S_l-I_l,R_l)\,\mathop{=}^{(d)}\,(M_l-1,K_l).$$ 
So Lemma \ref{1.11randomtrees} can be rephrased as follows.
Let $\epsilon \in (0,1/4)$. We can find $\epsilon'>0$ and $n_0 \geqslant 1$ such that for $m \geqslant n_0$ and $ l \in \lbrace 0, \dots, m \rbrace$, we have
\begin{equation}
\label{1.11reformule}
P_1\left( \left|\frac{S_l-I_l+1}{\sqrt{m}}- \frac{3 \sigma^2}{4} \frac{R_{l}} {\sqrt{m}} \right|> m^{-1/4+\epsilon} \right) < \exp\left( -m^{\epsilon'} \right).
\end{equation} 
Then, since the function $f_{\delta}$ is bounded and Lipschitz, we have
$$ \begin{aligned}
&\sqrt{n}\, E_1\!\left[\mathbf{1}_{\{\tau > \lceil (1-\delta)n \rceil\}} \left|f_\delta\!\left(\frac{S_{\lceil (1-\delta)n \rceil}}{\sqrt{m_n}} \right)-f_\delta\!\left( \frac{3 \sigma^2}{4}\frac{R_{\lceil (1-\delta)n \rceil}}{\sqrt{m_n}} \right)  \right| \right] \\
&\leq \sqrt{n}\, K_{\delta} \,E_1\! \left[\mathbf{1}_{\{\tau > \lceil (1-\delta)n \rceil\}} \left( \left| \frac{S_{\lceil (1-\delta)n \rceil}-\frac{3\sigma^2}{4} R_{\lceil (1-\delta)n \rceil}}{\sqrt{m_n}} \right| \wedge 1 \right)  \right]. \\
\end{aligned}$$
where the constant $K_\delta$ only depends on $\delta$. 
It follows that$$ \begin{aligned}
&\sqrt{n} \,E_1\!\left[\mathbf{1}_{\{\tau > \lceil (1-\delta)n \rceil\}} \left|f_\delta\!\left(\frac{S_{\lceil (1-\delta)n \rceil}}{\sqrt{m_n}} \right)-f_\delta\!\left( \frac{3 \sigma^2}{4}\frac{R_{\lceil (1-\delta)n \rceil}}{\sqrt{m_n}} \right)  \right| \right] \\
&\leq \sqrt{n}\, K_{\delta}\, \frac{1}{\sqrt{m_n}} n^{1/4+\epsilon}  E_1[\mathbf{1}_{\{\tau > \lceil (1-\delta)n \rceil\}}] \\
&+ \sqrt{n} \,K_{\delta} \,P_1 \left(\left| S_{\lceil (1-\delta)n \rceil}-\frac{3\sigma^2}{4} R_{\lceil (1-\delta)n \rceil}\right| > n^{1/4+\epsilon}, \tau > \lceil (1-\delta)n \rceil  \right). \\
\end{aligned}$$ 
The first term in the sum tends to $0$ as $n\to\infty$ thanks to \eqref{estimationpopu}. We then use the fact that $I_{\lceil (1-\delta)n \rceil}=1$ on the event $\{\tau > \lceil (1-\delta)n \rceil \}$ and the bound \eqref{1.11reformule} to
see that the second term also tends to $0$. We thus get
$$ \sqrt{n} \,E_1\!\left[\mathbf{1}_{\{\tau > \lceil (1-\delta)n \rceil\}} \left|f_\delta\!\left(\frac{S_{\lceil (1-\delta)n \rceil}}{\sqrt{m_n}} \right)-f_\delta\!\left( \frac{3 \sigma^2}{4}\frac{R_{\lceil (1-\delta)n \rceil}}{\sqrt{m_n}} \right)  \right| \right] \underset {n \to \infty} {\longrightarrow} 0$$
and our claim follows. 
\end{proof}

It follows from Lemmas \ref{cvgphiY} and \ref{cvgfCY} that
\begin{equation}
\label{cvgphiC}
\lim_{n\to\infty} \sqrt{n} \,E\! \left[ \left| n \Phi_n(Y_{\lceil (1-\delta) n \rceil})-f_\delta\! \left( \frac{3 \sigma^2}{4} \frac{C^{\mathcal{T}^0}_{\lceil (1-\delta)n \rceil}} {\sqrt{m_n}} \right) \right|   \mathbf{1}_{\{N \geq \lceil (1-\delta)n \rceil\}} \right]=0.
\end{equation}
From \eqref{expressiondebut} and \eqref{phi}, we now obtain
\begin{align}
\label{endfirststep}
&\lim_{n\to\infty}\Bigg|n^{3/2}E\!\left[\Psi \left( \left( \frac{1}{\sqrt{n}} C^{\mathcal{T}^0}_{nt}, \frac{1}{n^{1/4}} L^{\mathcal{T}^0}_{nt} \right), 0 \leq t \leq1-\delta \right) \mathbf{1}_{\lbrace N=n \rbrace} \right]\\
&- \sqrt{n}\,E\!\Bigg[\Psi \left( \left( \frac{1}{\sqrt{n}} C^{\mathcal{T}^0}_{nt}, \frac{1}{n^{1/4}} L^{\mathcal{T}^0}_{nt} \right), 0 \leq t \leq 1-\delta \right) f_\delta\! \left( \frac{3 \sigma^2}{4} \frac{C^{\mathcal{T}^0}_{\lceil (1-\delta)n \rceil}} {\sqrt{m_n}} \right)\mathbf{1}_{\{N \geq \lceil (1-\delta)n \rceil\}} \Bigg] \Bigg|=0.\notag
\end{align}

\subsection*{Second step}
In view of \eqref{endfirststep}, we now need to get a limit in distribution for the (rescaled) pair $(C^{\mathcal{T}^0}_{nt}, L^{\mathcal{T}^0}_{nt})_{0\leq t\leq 1-\delta}$ 
conditioned on the event $\{N \geq \lceil (1-\delta)n \rceil\}$. This is the goal of the next lemma, which is essentially a consequence of results found in 
\cite{Mtrees}. 

\begin{lemma}
\label{cvggregory}
Let $a>0$. The law under $P(. |N \geq an)$ of the process
$$ \left( \left(\frac{1}{\sqrt{n}} C^{\mathcal{T}^0}_{(nt)\wedge N}, \frac{1}{n^{1/4}} L^{\mathcal{T}^0}_{(nt)\wedge N} \right), t\geq 0 \right)$$ converges when $n \to \infty$ to the law of   $$\left(\left( \frac{1}{\tilde{\sigma}} \mathbf{e}^{(a)}_t, \Sigma \sqrt{\frac{2}{\tilde{\sigma}}} Z^{(a)}_t  \right), t\geq 0 \right)$$
where  $\mathbf{e}^{(a)}$ is a Brownian excursion conditioned to have duration greater than $a$, $Z^{(a)}$ is the Brownian snake driven by this excursion, and the constants are given by  
$$\tilde{\sigma}= \frac{9}{4 \sqrt{2}}\ ,\quad\Sigma= \sqrt{\frac{9}{8}}.$$
\end{lemma}

\begin{proof}
To relate the convergence of the lemma to the results of \cite{Mtrees}, we first recall the contour function $C^\mathcal{T}$  and introduce a label function $L^\mathcal{T}$ defined as follows. If $(u_0,u_1,\ldots,u_{2N})$ is the contour sequence of $\mathcal{T}$, we already saw that $C^\mathcal{T}_i=|u_i|$
and we put $L^\mathcal{T}_i=\ell(u_i)$, for every $i\in\{0,1,\ldots,2N\}$, where by convention we have assigned to each black vertex the label of its parent. 
We then interpolate linearly to define $C^\mathcal{T}_t$ and $L^\mathcal{T}_t$ for every real $t\in[0,2N]$.
It is then enough to verify that the convergence of the lemma holds 
when $(n^{-1/2}C^{\mathcal{T}^0}_{(nt)\wedge N}, n^{-1/4} L^{\mathcal{T}^0}_{(nt)\wedge N})_{t\geq 0}$ is replaced by $(2^{-1}n^{-1/2}C^{\mathcal{T}}_{(2nt)\wedge (2N)}, n^{-1/4} L^{\mathcal{T}}_{(2nt)\wedge (2N)})_{t\geq 0}$ (see Remark \ref{rqCt}).

We also introduce the
variant of the contour function called the height function, and the corresponding variant of the label function.
The height function of $\mathcal{T}$ is defined by setting $H^\mathcal{T}_i= |w_i|$ for $0\leq i\leq N$, where $w_0,w_1,\ldots,w_N$ are the vertices
of $\mathcal{T}$ listed in lexicographical order, and the modified label function is defined by $\tilde L^\mathcal{T}_i=\ell(w_i)$ (again  we assign to each black vertex the label of its parent).
By convention we set $H^\mathcal{T}_{N+1}=0$ and $\tilde L^\mathcal{T}_{N+1}=0$.
Both $H^\mathcal{T}$ and $\tilde L^\mathcal{T}$ are interpolated linearly to give processes indexed by $[0,N+1]$. Then we may
replace $(2^{-1}n^{-1/2}C^{\mathcal{T}}_{(2nt)\wedge (2N)}, n^{-1/4} L^{\mathcal{T}}_{(2nt)\wedge (2N)})_{t\geq 0}$ by 
$(2^{-1}n^{-1/2}H^{\mathcal{T}}_{(nt)\wedge (N+1)}, n^{-1/4} \tilde L^{\mathcal{T}}_{(nt)\wedge (N+1)})_{t\geq 0}$. Indeed it is well known that 
asymptotics for the height functions, of the type of the convergence \eqref{cvgforest}, imply similar asymptotics for the contour functions (and similarly
for the label functions) modulo an extra multiplicative factor $2$ in the time scaling. See e.g. Section 1.6 in \cite{randomtrees} for a precise justification in
a slightly different setting. In the case of Galton-Watson trees with a fixed size, the fact that the height process and the contour function converge jointly to the same Brownian excursion is due to Marckert and Mokkadem in \cite{MaMo}.

Consider then a sequence $(\mathcal{T}_{(k)},(\ell_{(k)}(u))_{u\in\mathcal{T}^0_{(k)}})_{k\geq 1}$ of independent labeled trees distributed as $(\mathcal{T},(\ell(u))_{u\in\mathcal{T}^0})$. Set $N_{(k)}=|\mathcal{T}_{(k)}|$ for every $k\geq 1$.
Define the height function $H^\infty$, respectively the label function $\tilde L^\infty$, by concatenating the height functions $(H^{\mathcal{T}_{(k)}}_t)_{0\leq t\leq N_{(k)}+1}$,
resp. the label functions $(\tilde L^{\mathcal{T}_{(k)}}_t)_{0\leq t\leq N_{(k)}+1}$. Then a very special case of Theorems 1 and 3 in \cite{Mtrees} gives the
convergence in distribution
\begin{equation}
\label{cvgforest}
\left( \left(\frac{1}{\sqrt{n}} H^{\infty}_{nt}, \frac{1}{n^{1/4}} \tilde L^{\infty}_{nt} \right), t\geq 0 \right)
\build{\longrightarrow}_{n\to\infty}^{(d)}\left(\left( \frac{2}{\tilde{\sigma}} \beta_t, \Sigma \sqrt{\frac{2}{\tilde{\sigma}}} W_t  \right), t\geq 0 \right)
\end{equation}
where $\beta$ is a standard reflected linear Brownian motion, and $W$ is the Brownian snake driven by $\beta$. Furthermore, the constants $\tilde\sigma$ and $\Sigma$ are as in the statement of the lemma. 

Let us comment on the numerical values of the constants $\tilde\sigma$ and $\Sigma$. Both these constants can be calculated using the 
formulas found in \cite{Mtrees}. More precisely, $\tilde\sigma$ is evaluated from formula (2) in \cite{Mtrees}, using also the numerical
values $\sigma_0^2=3/4$ and $\sigma_1^2=15/2$ for the respective variances of $\mu_0$ and $\mu_1$. Similarly, $\Sigma$ is computed from
the formula in \cite[Theorem 3]{Mtrees}. When applying this formula, we need to calculate the variance of the difference 
between the label of the $i$-th child of a black vertex and the label of the parent of this black vertex, conditionally on the 
event that the black vertex in consideration has $p$ children (with of course $p\geq i$).  This variance is equal to $2i(p-i+1)/(p+2)$, by
a calculation found on page 1664 of \cite{MM}. The remaining part of the calculation is straightforward, and we leave the
details to the reader.

Finally we observe that if $K=\min\{k\geq 1: N_{(k)}\geq an\}$, the law of the labeled tree $(\mathcal{T}_{(K)},(\ell_{(K)}(u))_{u\in\mathcal{T}^0_{(K)}})$ is the same as
the conditional law  $(\mathcal{T},(\ell(u))_{u\in\mathcal{T}^0})$ knowing that $N\geq an$. On the other hand, the process $(n^{-1/2} H^{\mathcal{T}_{(K)}}_{nt})_{0\leq t\leq n^{-1}(N_{(K)}+1)}$ corresponds to the first excursion of $(n^{-1/2}H^{\infty}_{nt})_{t\geq 0}$ away from $0$ with length greater than or equal to $a+n^{-1}$. 
By arguments very similar to \cite[Proof of Corollary 1.13]{randomtrees}, we deduce from \eqref{cvgforest} that $(n^{-1/2} H^{\mathcal{T}_{(K)}}_{(nt)\wedge (N_{(K)}+1)})_{t\geq 0}$
converges in distribution to the first excursion of $( \frac{2}{\tilde{\sigma}} \beta_t)_{t\geq 0}$ away from $0$ with duration greater than $a$. This gives the convergence of the first component in Lemma \ref{cvggregory}. The convergence of the second component (and the fact that it holds jointly with the first one) is obtained by the same
argument. 
\end{proof}

By \eqref{estimationpopu}, we have
$$n^{3/2}P(N=n)\build{\longrightarrow}_{n\to\infty}^{} \frac{1}{\sigma\sqrt{2\pi}}\ ,\quad\sqrt{n}\,P(N\geq \lceil (1-\delta)n\rceil) \build{\longrightarrow}_{n\to\infty}^{} \frac{2}{\sigma \sqrt{2\pi}}\,(1-\delta)^{-1/2}.$$
From \eqref{endfirststep} and Lemma \ref{cvggregory}, we now get
\begin{align*}
&\lim_{n\to\infty}E\!\left[\Psi \left( \left( \frac{1}{\sqrt{n}} C^{\mathcal{T}^0}_{nt}, \frac{1}{n^{1/4}} L^{\mathcal{T}^0}_{nt} \right), 0 \leq t \leq1-\delta \right) \,\Bigg|\, N=n \right]\\
&=2(1-\delta)^{-1/2}E\!\Bigg[\Psi \left( \left(\frac{1}{\tilde{\sigma}} \mathbf{e}^{(1-\delta)}_t, \Sigma \sqrt{\frac{2}{\tilde{\sigma}}} Z^{(1-\delta)}_t\right)  , 0\leq t\leq 1-\delta \right)
f_\delta\!\left(\frac{3}{\sqrt{2\delta}} \mathbf{e}^{(1-\delta)}_{1-\delta}\right)\Bigg]\\
&=E\!\Bigg[\Psi \left( \left(\frac{1}{\tilde{\sigma}} \mathbf{e}^{(1-\delta)}_t, \Sigma \sqrt{\frac{2}{\tilde{\sigma}}} Z^{(1-\delta)}_t 
\right) , 0\leq t\leq 1-\delta \right)
g_\delta\!\left(\mathbf{e}^{(1-\delta)}_{1-\delta}\right)\Bigg],
\end{align*}
where, for every $x\geq 0$,
$$g_\delta(x)= 2(1-\delta)^{-1/2}\,f_\delta\!\left(\frac{3}{\sqrt{2\delta}}\,x\right).$$
Recalling the definition of $f_\delta$, and the fact that $\sigma^2=9/2$, we obtain
$$g_\delta(x)= \frac{2x}{\sqrt{2 \pi \delta^3(1-\delta)}} \exp\left(- \frac{x^2}{2 \delta} \right).$$
It is well known (see formula (1) in \cite{Ito}) that the function $\omega\longrightarrow g_\delta(\omega(1-\delta))$ is the density (on the space $\mathcal{C}(\R_+,\R_+)$) of the
law of the normalized Brownian excursion  with respect to the law of the Brownian excursion conditioned to have length greater than $1-\delta$, 
on the $\sigma$-field generated by the coordinates up to time $1-\delta$. Hence we conclude that we have also
\begin{align}
\label{cvgtheor}
&\lim_{n\to\infty}E\!\left[\Psi \left( \left( \frac{1}{\sqrt{n}} C^{\mathcal{T}^0}_{nt}, \frac{1}{n^{1/4}} L^{\mathcal{T}^0}_{nt} \right), 0 \leq t \leq1-\delta \right) \,\Bigg|\, N=n \right]\\
&=E\!\Bigg[\Psi \left(\left( \frac{1}{\tilde{\sigma}} \mathbf{e}_t, \Sigma \sqrt{\frac{2}{\tilde{\sigma}}} Z_t \right) , 0\leq t\leq 1-\delta \right)
\Bigg],\notag
\end{align}
where $\mathbf{e}$ and $Z$ are as in the statement of Theorem \ref{cvgcontourlabel}. Since this holds for every $\delta\in(0,1)$ and since
we have $C^{\mathcal{T}^0}_{n}=L^{\mathcal{T}^0}_{n}=0$ on the event $\{N=n\}$, we have obtained the convergence of finite-marginal
distributions in the convergence of Theorem \ref{cvgcontourlabel} (note that $ \frac{1}{\tilde{\sigma}}=\frac{4\sqrt{2}}{9}$ and $\Sigma \sqrt{\frac{2}{\tilde{\sigma}}}=2^{1/4}$).

To complete the proof, we still need a tightness argument. But tightness holds if we restrict our processes to $[0,1-\delta]$ by
\eqref{cvgtheor}, and we can then use a time-reversal argument. Indeed $(C^{\mathcal{T}^0}_0,C^{\mathcal{T}^0}_1,\ldots,
C^{\mathcal{T}^0}_n)$ and  $(C^{\mathcal{T}^0}_n,C^{\mathcal{T}^0}_{n-1},\ldots,
C^{\mathcal{T}^0}_0)$ have the same distribution under $P(\cdot\,|\, N=n)$. The similar property does not
hold for the label process, but  $(L^{\mathcal{T}^0}_n,L^{\mathcal{T}^0}_{n-1},\ldots,
L^{\mathcal{T}^0}_0)$ corresponds to the label process for a (conditioned) tree where labels would be
generated by using the counterclockwise order instead of 
the clockwise order, in the constraints of the definition of a labeled tree in subsection 2.1. Clearly, our
arguments would go through with this different convention, and so we get the desired tightness also for the
label process. This completes the proof of Theorem \ref{cvgcontourlabel}. 

\endproof

\begin{remark}
The difficulty in proving Theorem \ref{cvgcontourlabel} comes from the convergence of labels. If we had been interested only in the convergence of the rescaled contour functions $ \left(\frac{1}{\sqrt{n}} C^{\mathcal{T}^0}_{nt}\right)$, we could have used formula \eqref{lienCY} more directly, following the ideas of Marckert and Mokkadem \cite{MaMo}. See also \cite[Chapter 1]{randomtrees}.
\end{remark}

\section{Convergence towards the Brownian map for rooted and pointed maps}
\label{CvgCarte}

Recall that $\mathcal{M}_n^{\bullet}$ is a random bipartite planar map uniformly distributed over the set $\mathbf{M}^{b \bullet}_n$ of all bipartite planar rooted and pointed maps with $n$ edges. In this section, we prove the analog of Theorem \ref{cvgcarte} when $\mathcal{M}_n$ is replaced by $\mathcal{M}_n^{\bullet}$, namely
\begin{equation}
\label{cvgcartepointee}
\left(V(\mathcal{M}_n^{\bullet}) , 2^{-1/4} n^{-1/4} d_{\rm{gr}}^{\mathcal{M}_n^{\bullet}}\right)  \build{\longrightarrow}_{n \to \infty}^{(d)} (\mathbf{m}_{\infty}, D^*) 
\end{equation}
 where $(\mathbf{m}_{\infty}, D^*)$ is the Brownian map.

\subsection{Definition of the Brownian map}
We define the Brownian map following \cite[Sect.2.4]{uniqueness}.
We first need to introduce the CRT (Continuous Real Tree). Let $(\mathbf{e}_s)_{0 \leq s \leq 1}$ be a normalized Brownian excursion. For $s, t \in [0,1]$, we set
$$ d_{\mathbf{e}}(s,t)=\mathbf{e}_s+\mathbf{e}_t-2 \min \{ \mathbf{e}_r: s \wedge t \leq r \leq s \vee t \}.$$
We notice that $d_{\mathbf{e}}$ is a random pseudo-metric on $[0,1]$. Consider the equivalence relation defined for $s, t \in [0,1]$ by
$$ s \sim_{\mathbf{e}} t \ \text{iff} \ d_{\mathbf{e}}(s,t)=0.$$
The CRT is then the quotient space  $\mathcal{T}_{\mathbf{e}}= [0,1] /\sim_{\mathbf{e}}$, which is equipped with the distance induced by $d_{\mathbf{e}}$.
We denote the canonical projection $[0,1] \rightarrow \mathcal{T}_{\mathbf{e}}$ by $p_{\mathbf{e}}$.

We then let $Z=(Z_s)_{0 \leq s \leq 1}$ be the Brownian snake
driven by $\mathbf{e}$, as in Theorem \ref{cvgcontourlabel}. We note that $Z_0=0$ and $E((Z_s-Z_t)^2 |\mathbf{e})=d_{\mathbf{e}}(s,t)$. 
From the last relation, one obtains that $Z_s=Z_t$ for every $s,t \in [0,1]$ such that $d_{\mathbf{e}}(s,t)=0$, a.s. Thus the process $Z$ can be viewed as indexed by the CRT $\mathcal{T}_{\mathbf{e}}$, in such a way that $Z_s=Z_{p_{\mathbf{e}(s)}}$ for $ s \in [0,1]$. In the sequel, we will use the notation $Z_s=Z_a$ if $s \in [0,1]$ and $a=p_{\mathbf{e}}(s)$. 
Using similar techniques as in the proof of the Kolmogorov regularity theorem, one can show that the mapping $a \mapsto Z_a$ is H\"older continuous with exponent  $\frac{1}{2}-\epsilon$ with respect to $d_{\mathbf{e}}$, for every $\epsilon \in ]0, \frac{1}{2}[$. The pair $(\mathcal{T}_{\mathbf{e}}, (Z_a)_{a \in \mathcal{T}_{\mathbf{e}}})$ is then a continuous analog of discrete labeled trees. 

We can now define the Brownian map, as a quotient space of the CRT. For $s,t \in [0,1]$ such that $s \leq t$, we set
$$ D^0(s,t)=D^0(t,s)=Z_s+Z_t-2\max \left(\min \{Z_r, r \in [s,t] \}, \min \{ Z_r, r \in [0,s] \cup [t,1] \}\right)$$
and for $a, b \in \mathcal{T}_{\mathbf{e}}$,
$$D^0(a,b)= \min \{D^0(s,t): (s,t) \in [0,1]^2, p_{\mathbf{e}}(s)=a, p_{\mathbf{e}}(t)=b \}.$$
Finally, for $a, b \in \mathcal{T}_{\mathbf{e}}$, let
$$D^*(a,b)= \inf \left\lbrace \sum_{i=1}^k D^0(a_{i-1},a_i) \right\rbrace$$  
where the infimum is over all choices of the integer $k \geq 1$ and of the finite sequence $(a_0, \dots a_k)$ of elements of $\mathcal{T}_{\mathbf{e}}$ such that $a_0=a$ and $a_k=b$. Then, $D^*$ is a pseudo-metric on the CRT $\mathcal{T}_{\mathbf{e}}$, which satisfies  $D^* \leq D^0$. One can also interpret $D^*$ as a function on $[0,1]^2$ by setting $D^*(s,t)=D^*(p_{\mathbf{e}}(s), p_{\mathbf{e}}(t))$ for $(s,t) \in [0,1]^2$. Let  $\simeq$ be the equivalence relation on $\mathcal{T}_{\mathbf{e}}$ given by
$$ a \simeq b \ \text{iff} \  D^*(a,b)=0.$$ 
We set
$$\mathbf{m}_{\infty}=\mathcal{T}_{\mathbf{e}} / \simeq$$
and let $\Pi : \mathcal{T}_{\mathbf{e}} \rightarrow \mathbf{m}_{\infty}$ be the canonical projection. The Brownian map is the space $\mathbf{m}_{\infty}$ equipped with the distance induced by $D^*$. 

\subsection{Proof of the convergence towards the Brownian map}
As previously, we let $(\mathcal{T}_n, (\ell_n(v))_{v \in \mathcal{T}^0_n})$ be the random  labeled tree associated with $\mathcal{M}_n^{\bullet}$ via the BDG bijection. Recall that $\mathcal{T}_n$ is a two-type Galton-Watson tree with offspring distributions $\mu_0$ and $\mu_1$, conditioned to have $n$ edges.  We use the notation $(v^n_0, \dots, v^n_n)$ for the white contour sequence of $\mathcal{T}_n$. Recall that the white vertices in $\mathcal{T}_n$ are identified to vertices of the map $\mathcal{M}_n^{\bullet}$. For $ (i,j) \in \{0, \dots n \}^2$, we set
$$ d_n(i,j)=d_{\text{gr}}^{\mathcal{M}_n^{\bullet}}(v^n_i,v^n_j).$$
We then extend this definition to noninteger values of $i$ and $j$ by putting for $s,t \in [0,n]^2$
$$\begin{aligned}
 d_n(s,t)= &(s-\lfloor s \rfloor) (t-\lfloor t \rfloor) d_n(\lceil s \rceil, \lceil t \rceil)+  (s-\lfloor s \rfloor) (\lceil t \rceil- t ) d_n(\lceil s \rceil, \lfloor t \rfloor) \\
 +& (\lceil s \rceil- s ) (t-\lfloor t \rfloor) d_n(\lfloor s \rfloor, \lceil t \rceil)+ (\lceil s \rceil - s ) (\lceil t \rceil-t ) d_n(\lfloor s \rfloor, \lfloor t \rfloor). \\
\end{aligned} $$
Recall our convention $v^n_{n+i}=v^n_i$ for $0 \leq i \leq n$.
From the bound \eqref{majolabeldistance}, we have for $0 \leq i < j \leq n$, 
\begin{align}
\label{bounddistance}
d_n(i,j)&\leq \ell_n(v^n_i) + \ell_n(v^n_j)-2 \max \lbrace \min \lbrace \ell_n(v^n_k), i \leq k \leq j \rbrace, \min \lbrace \ell_n(v^n_k), j \leq k \leq i+n \rbrace \rbrace +2\\
&=L_{i}^{\mathcal{T}^0_n}+L_{j}^{\mathcal{T}^0_n}-2\max \lbrace \min\lbrace L_{k}^{\mathcal{T}^0_n}, k\in[i,j]\rbrace, \min \lbrace L_{k}^{\mathcal{T}^0_n}, k\in[j,n]\cup[0,i]\rbrace\rbrace
\notag
\end{align}
From the last bound and the convergence in distribution of the sequence of processes  $(n^{-1/4}  L_{nt}^{\mathcal{T}^0_n})_{0 \leq t \leq 1}$ (Theorem \ref{cvgcontourlabel}), one gets that the sequence of the distributions of 
the processes
$$\Big( n^{-1/4}d_n(ns,nt),0 \leq s, t \leq 1\Big)$$ 
is tight. Using Theorem \ref{cvgcontourlabel} and Remark \ref{rqCt}, we see that we can find a sequence $(n_k)_{k \geq 1}$ tending to infinity and a continuous random process $(D(s,t))_{0 \leq s,t \leq 1}$ such that, along $(n_k)_{k \geq 1}$, the following joint convergence in distribution in $\mathcal{C}([0,1]^2, \R^3)$ holds:
\begin{equation}
\label{cvgtriplet}
\left(\frac{9}{8\sqrt{2}}\, \frac{C^{\mathcal{T}_n}_{2nt}}{n^{1/2}}, \, 2^{-1/4} \frac{L^{\mathcal{T}^0_n}_{nt}}{n^{1/4}}, \,2^{-1/4}  \frac{d_n(s,t)}{n^{1/4}} \right)_{0 \leq s,t \leq 1} \underset{ n\to \infty} {\longrightarrow} (\mathbf{e}_t,  Z_t, D(s,t))_{0 \leq s,t \leq 1}.
\end{equation}
Using the Skorokhod representation theorem (and recalling that  $(\mathcal{T}_n, (\ell_n(v))_{v \in \mathcal{T}^0_n})$
is determined by the pair $(C^{\mathcal{T}_n}, L^{\mathcal{T}^0_n})$), we may and will assume that the convergence \eqref{cvgtriplet} holds a.s. 
along the sequence $(n_k)_{k\geq 1}$. 
From the definition of $D^0(s,t)$ and the bound \eqref{bounddistance}, we obtain that
for every $ (s,t) \in [0,1]^2$, 
\begin{equation}
\label{boundDDzero}
D(s,t) \leq D^0(s,t)
\end{equation}
Similarly, a passage to the limit from the identity \eqref{lienlabeldistance} gives
\begin{equation}
\label{loiD}
D(0,t)=Z_t - \min\{Z_s:0\leq s\leq 1\},
\end{equation} for every $t \in [0,1]$, a.s.

The function $(s,t) \mapsto D(s,t)$ is clearly symmetric and satisfies the triangle inequality since the functions $d_n$ do. Moreover, the fact that 
$d_n(i,j)=0$ if $v^n_i=v^n_j$ easily implies that $D(s,t)=0$ for $s,t$ such that $s \sim_{\mathbf{e}} t$ a.s. (see the proof of
Proposition 3.3 in \cite{topostructure} for a similar
argument). Hence $D(s,t)$ only depends on $p_\mathbf{e}(s)$ and $p_\mathbf{e}(t)$, and  $D$ can be viewed as a pseudo-metric on the CRT  $\mathcal{T}_{\mathbf{e}}$, which satisfies $D(a,b) \leq D^0(a,b)$ for every $a,b \in \mathcal{T}_{\mathbf{e}}$, by
\eqref{boundDDzero}. 
Since $D$ verifies the triangle inequality, the latter bound also implies
$$ D(a,b) \leq D^{*}(a,b)$$ for every $ a,b \in \mathcal{T}_{\mathbf{e}}$ a.s.
To complete the proof, we need the next lemma.

\begin{lemma}
\label{DegalDetoile}
We have $$D(a,b)=D^*(a,b)$$ for every $a,b \in \mathcal{T}_{\mathbf{e}}$ a.s.
\end{lemma}

The statement of the theorem easily follows from the lemma. Indeed, we introduce a correspondence between the 
metric spaces $(V(\mathcal{M}_n^{\bullet}) \setminus \{\partial\}, 2^{-1/4} n^{-1/4} d_{\text{gr}}^{\mathcal{M}_n^{\bullet}})$ and  $(\mathbf{m}_{\infty}, D^*)$ by setting
$$\mathcal{R}_n=\{(v^n_{\lfloor nt \rfloor}, \Pi(p_{\mathbf{e}}(t))) : t \in [0,1] \}.$$
From the (almost sure) convergence \eqref{cvgtriplet}, and the equality $D=D^*$, we easily get that the distortion of $\mathcal{R}_n$
tends to $0$ as $n\to\infty$  along the sequence $(n_k)_{k \geq 1}$. It follows that the random metric space $(V(\mathcal{M}_n^{\bullet}) \setminus \{\partial\}, 2^{-1/4} n^{-1/4} d_{\text{gr}}^{\mathcal{M}_n^{\bullet}})$ converges a.s. to  $(\mathbf{m}_{\infty}, D^*)$ as $n\to\infty$ along the sequence $(n_k)_{k \geq 1}$, in the Gromov-Hausdorff sense.
Clearly, this convergence still holds if we replace $V(\mathcal{M}_n^{\bullet}) \setminus \{ \partial \}$ by $V(\mathcal{M}_n^{\bullet})$.
The previous discussion shows that from every sequence of integers going to infinity, we can extract a subsequence along which the convergence stated in
 \eqref{cvgcartepointee} holds. This suffices to complete the proof of   \eqref{cvgcartepointee}.

It only remains to prove Lemma \ref{DegalDetoile}. 

\subsection{Proof of Lemma \ref{DegalDetoile}}
Here we follow closely \cite[Section 8.3]{uniqueness}.
By a continuity argument, it is enough to show that if $X$ and $Y$ are two independent random variables uniformly distributed over $[0,1]$, which are also independent of 
the sequence $(\mathcal{M}_n^{\bullet})_{n\geq 1}$ and of the triplet $(\mathbf{e}, Z, D)$, we have
$$ D(p_{\mathbf{e}}(X), p_{\mathbf{e}}(Y))= D^*(p_{\mathbf{e}}(X), p_{\mathbf{e}}(Y)) \ \text{a.s}.$$ 
Since one already knows that $$ D(p_{\mathbf{e}}(X), p_{\mathbf{e}}(Y))\leq D^*(p_{\mathbf{e}}(X), p_{\mathbf{e}}(Y)),$$
it is enough to prove that these two random variables have the same distribution. 

First, the distribution of $D^*(p_{\mathbf{e}}(X), p_{\mathbf{e}}(Y))$ can be found in \cite[Corollary 7.3]{uniqueness}:
\begin{equation}
\label{lawD*}
D^*(p_{\mathbf{e}}(X), p_{\mathbf{e}}(Y))\mathop{=}^{(d)} Z_X-\min\{Z_s:0\leq s\leq 1\}.
\end{equation}
We then want to determine the distribution of  $D(p_{\mathbf{e}}(X), p_{\mathbf{e}}(Y))=D(X,Y)$. We set for $n \geq 1$, 
$$i_n= \lfloor nX \rfloor \ , \ j_n=\lfloor nY \rfloor.$$
The random variables $i_n$ and $j_n$ are independent, independent of $\mathcal{M}_n^{\bullet}$ and uniformly distributed over $\{0, \dots, n-1 \}$. As we already explained in
subsection 2.2, every integer between $0$ and $n-1$ corresponds to a corner of a white vertex in the tree $\mathcal{T}_n$, and thus by the BDG bijection to an edge of $\mathcal{M}_n^{\bullet}$. 
We introduce a new planar map $\mathcal{M}_n^{\bullet '}$ in $\mathbf{M}_n^{b \bullet}$  defined by saying that $\mathcal{M}_n^{\bullet '}$ has the same vertices, edges, faces and origin vertex as $\mathcal{M}_n^{\bullet}$, but  a different root edge, which is the edge associated with the corner corresponding to $i_n$ in the BDG bijection between $\mathcal{T}_n$ and $\mathcal{M}_n^{\bullet}$. The orientation of this root edge is chosen with probability $\frac{1}{2}$ among the two possible ones. Since what we have done is just replacing the root edge by another oriented edge chosen uniformly
at random over the $2n$ possible choices, it is easy to see
that the map $\mathcal{M}_n^{\bullet '}$ is also uniformly distributed over $\mathbf{M}_n^{b \bullet}$.

The tree associated with $\mathcal{M}_n^{\bullet '}$ via the BDG bijection is denoted by $\mathcal{T}'_n$. We let $v'^{n}_0, \dots, v'^{n}_n$ be the white contour sequence of $\mathcal{T}'_n$ and 
we also let $d'_n$ be the analog of $d_n$ when $\mathcal{M}_n^{\bullet}$ is replaced by $\mathcal{M}_n^{\bullet '}$. 

Let $k_n \in \{0, \dots, n-1 \}$ be the index of the white corner of $\mathcal{T}'_n$ corresponding via the BDG bijection to the edge of $\mathcal{M}_n^{\bullet}$ starting from the corner $j_n$ in $\mathcal{T}_n$. 
Conditionally on the pair $(\mathcal{M}_n^{\bullet}, \mathcal{M}_n^{\bullet '})$, the latter edge is uniformly distributed over the set of all edges of $\mathcal{M}_n^{\bullet}$ (thus also over the set of all edges of $\mathcal{M}_n^{\bullet '}$). It follows that, conditionally to $(\mathcal{M}_n^{\bullet}, \mathcal{M}_n^{\bullet '})$, the index $k_n$ is uniformly distributed over $\{0, \dots, n-1 \}$, so it is independent of $\mathcal{M}_n^{\bullet '}$. 
From the definition of $\mathcal{M}_n^{\bullet '}$, the vertex $v^n_{i_n}$ is either equal or adjacent to $v'^n_0$ and in a similar way the vertex $v^n_{j_n}$ is either equal or adjacent to $v'^n_{k_n}$. This leads to the bound.
\begin{equation}
\label{majoddprime}
| d_n(i_n,j_n)-d'_n(0,k_n) | \leq 2.
\end{equation}
Moreover we observe that
\begin{equation}
\label{lienddprime}
d'_n(0,k_n)\build{=}_{}^{(d)}d_n(0,i_n)
\end{equation}
because $k_n$ is independent of $\mathcal{M}_n^{\bullet '}$ and uniformly distributed over $\{0, \dots, n-1 \}$, and $i_n$ satisfies the same properties with respect to $\mathcal{M}_n^{\bullet}$.
We now use the a.s. convergence \eqref{cvgtriplet} to get 
\begin{equation}
\label{cvg1}
2^{-1/4} n^{-1/4} d_n(0,i_n) \underset {n \to \infty} {\longrightarrow} D(0,X) = Z_X-\min\{Z_s:0\leq s\leq 1\},
\end{equation}
where the last equality holds by \eqref{loiD}, and 
\begin{equation}
\label{cvg2}
2^{-1/4} n^{-1/4} d_n(i_n,j_n) \underset {n \to \infty} {\longrightarrow} D(X,Y). 
\end{equation}
Both \eqref{cvg1} and \eqref{cvg2} hold a.s along the subsequence $(n_k)_{k \geq 1}$.
On the other hand, \eqref{majoddprime} and \eqref{lienddprime} show that the limit in \eqref{cvg1}
must have the same distribution as the limit in  \eqref{cvg2}, and we get 
$$D(X,Y) \mathop{=}^{(d)} Z_X-\min\{Z_s:0\leq s\leq 1\}.$$
Recalling \eqref{lawD*}, we see that $ D(p_{\mathbf{e}}(X), p_{\mathbf{e}}(Y))$ and $ D^*(p_{\mathbf{e}}(X), p_{\mathbf{e}}(Y))$ have the same distribution, which completes the proof of 
Lemma \ref{DegalDetoile}.

\section{Convergence of rooted maps} 
In this section, we derive Theorem \ref{cvgcarte} from the convergence
\eqref{cvgcartepointee} for rooted and pointed maps. 
Notice that similar arguments appear in \cite[Proposition 4]{BJM}. 
As previously, $\mathcal{M}_n^{\bullet}$ is uniformly distributed over $\mathbf{M}_n^{b \bullet}$,
but it will be sometimes be convenient to view $\mathcal{M}_n^{\bullet}$ as a random element
of $\mathbf{M}_n^{b}$, just by ``forgetting'' the distinguished vertex. In particular, if
$F$ is a function on $\mathbf{M}_n^{b}$, the notation $F(\mathcal{M}_n^{\bullet})$ means that
we apply $F$ to the rooted map obtained by forgetting the distinguished vertex of
$\mathcal{M}_n^{\bullet}$. Similarly, we will write $\mu_n^{\bullet}$ for the law of
$\mathcal{M}_n^{\bullet}$ viewed as a random element
of $\mathbf{M}_n^{b}$. The notation $\mu_n$ will then stand for the law of $\mathcal{M}_n$, that is,
the uniform probability measure on  $\mathbf{M}^{b}_n$.
Let $\Vert . \Vert$ stand for the total variation norm. In order to get Theorem \ref{cvgcarte} from \eqref{cvgcartepointee}, it is sufficient to prove the following result.

\begin{proposition}
\label{variationtotale}
The following convergence holds. 
$$ \Vert \mu_n- \mu_n^{\bullet} \Vert \underset{n \to \infty} {\longrightarrow} 0.$$
\end{proposition}

\begin{proof}
We have $$\Vert \mu_n- \mu_n^{\bullet} \Vert = \frac{1}{2} \underset{-1 \leq F \leq 1} {\sup} |E(F(\mathcal{M}_n))-E(F(\mathcal{M}_n^{\bullet})) |,$$
where the supremum is over all functions $F:\mathbf{M}_n^{b \bullet} \longrightarrow [-1,1]$.
The quantity $E(F(\mathcal{M}_n^{\bullet}))$ can be expressed in terms of $E(F(\mathcal{M}_n))$ as
$$E(F(\mathcal{M}_n^{\bullet}))= \frac{E(F(\mathcal{M}_n) \text{Card}\, V(\mathcal{M}_n))}{E(\text{Card}\,V(\mathcal{M}_n))},$$
which implies 
\begin{equation}
\label{E.}
E(F(\mathcal{M}_n))= E\left(\frac{F(\mathcal{M}_n^{\bullet})}{\text{Card} \,V(\mathcal{M}_n^{\bullet})}\right) \frac{1}{E(1/\text{Card} \,V(\mathcal{M}_n^{\bullet}))}.
\end{equation}
We then need an estimate of $\text{Card}\,V(\mathcal{M}_n^{\bullet})$, which is given by the next lemma.

\begin{lemma}
\label{majoexpo}
Let $\delta >0$. There exists a positive constant $C_{\delta}$ such that 
$$ P \left(\left|{\rm Card}\, V(\mathcal{M}_n^{\bullet})- \frac{2n}{3}\right| > \delta n  \right)\leq \exp(-C_{\delta}n)$$ for all $n$ sufficiently large. 
\end{lemma}
\begin{proof}
We start by observing that the number $\text{Card}\, V(\mathcal{M}_n^{\bullet})$ corresponds via the BDG bijection to ($1$ plus) the number of white vertices of a two-type Galton-Watson tree with offspring distributions $(\mu_0, \mu_1)$ given by Proposition \ref{loiarbre}, conditioned to have $n$ edges.  

Let us consider a sequence of independent two-type Galton-Watson trees with offspring distributions $(\mu_0, \mu_1)$.
Suppose that the white vertices of these trees are listed in lexicographical order for each tree, one tree
after another,  and write $A_1, A_2,\dots$ for the respective numbers of black children of 
 the white vertices in this enumeration. Then $A_1,A_2,\ldots$ are i.i.d random variables with distribution $\mu_0$, and we recall that $\mu_0$ is a geometric distribution with mean $\frac{1}{2}$. We can apply Cramer's theorem to get the exponential bound, for every $n\geq 1$,
\begin{equation}
\label{majoexpo1}
  P\left(\left| \frac{A_1+\dots+A_{n}}{n}-\frac{1}{2} \right| \geqslant \delta \right) \leqslant \exp(-K_{\delta}n)
\end{equation} 
where $K_\delta$ is a positive constant. 

Let $N_0$ and $N_1$ be respectively the numbers of white and black vertices
in the first tree in our sequence, and let $N=N_0+N_1-1$, which is the number of edges of this tree.
The point now is the fact that if we condition on the event $\{N=n\}$,  the planar map
associated with the first tree becomes uniform on $\mathbf{M}_n^{b \bullet}$. Since this planar map
has  $N_0+1$ vertices, the result of the lemma will follow if we can prove that, for $n$ sufficiently large,
$$P\Big[\Big| N_0-\frac{2}{3}(n+1) \Big| > \delta(n+1)\,\Big| \, N=n\Big]\leq \exp(-C_\delta n)$$
for some positive constant $C_\delta$.

Recall from \eqref{estimationpopu}   that $n^{3/2}P(N=n) {\longrightarrow}({\sigma\sqrt{2\pi}}) ^{-1}$
as $n\to\infty$. Therefore the preceding exponential bound will follow if we can verify that
for all $n$ large enough,
$$P\Big[\Big\{\Big| N_0-\frac{2}{3}(n+1) \Big| > \delta(n+1)\,\Big|\Big\} \cap \{N=n\}\Big]\leq \exp(-c_\delta n)$$
with some positive constant $c_\delta$.

We first observe that the event $\mathcal{E}_1:=\lbrace  N_0-\frac{2}{3}(n+1)>\delta(n+1) \rbrace \cap \left\lbrace N=n \right\rbrace$ 
is contained in $$\left\lbrace n+1 \geqslant N_0 > \left(\frac{2}{3}+\delta \right)(n+1) \right\rbrace \cap \left\lbrace \frac{N_1}{N_0} < \frac{\frac{1}{3}-\delta}{\frac{2}{3}+\delta} \right\rbrace.$$
Therefore if we set  $a_{\delta}=(\frac{1}{3}-\delta)/(\frac{2}{3}+\delta) <\frac{1}{2}$, the event
$\mathcal{E}_1$ may only hold if, for some $k$ such that $(\frac{2}{3}+\delta)(n+1)<k\leq n+1$, 
the first $k$ white vertices of our sequence of trees have less than $a_\delta k$ black children.
Using \eqref{majoexpo1}, we obtain that
$$P(\mathcal{E}_1)\leq \sum_{(\frac{2}{3}+\delta)(n+1)<k\leq n+1} \exp(-K'_\delta k) \leq \exp(-c'_\delta n)$$
for some positive constants $K'_\delta$ and $c'_\delta$. Similar arguments give an analogous
exponential bound for the probability of the event
$\mathcal{E}_2:=\lbrace N_0-\frac{2}{3}(n+1)< -\delta(n+1) \rbrace \cap \lbrace N=n \rbrace$. This completes the proof of the lemma.
\end{proof}

Set $X_n=(2n/3)^{-1}\,\text{Card}\, V(\mathcal{M}_n^{\bullet})$ for every $n\geq 1$.

\begin{lemma}
\label{cvgL1}
The random variables $X_n^{-1}$ converge to $1$ in $L^1$ when $n$ tends to infinity.
\end{lemma}

\begin{proof}
First, as $\text{Card}\,V(\mathcal{M}_n^{\bullet}) \geq 1$, we have $X_n^{-1} \leq \frac{2n}{3}$. Let $\delta>0$.
The event $\lbrace |X_n^{-1}-1 | > \delta \rbrace$ is contained in $\lbrace X_n < \frac{1}{2} \rbrace \cup \lbrace |X_n-1 | > \frac{\delta}{2} \rbrace$. This leads to 
$$ E(|X_n^{-1}-1|) \leq \delta + E(|X_n^{-1}-1| \mathbf{1}_{\lbrace |X_n^{-1}-1| > \delta \rbrace}) \leq \delta+\frac{2n}{3}P\Big(|X_n-1|> \frac{\delta}{2}\wedge \frac{1}{2}\Big).$$ 
Hence, by Lemma \ref{majoexpo},
$$\limsup_{n\to\infty} E(|X_n^{-1}-1|)\leq \delta$$
and the desired result follows since $\delta$ was arbitrary.
\end{proof}

Finally we use \eqref{E.} and Lemma \ref{cvgL1} to get
$$\begin{aligned}
\Vert \mu_n-\mu_n^{\bullet} \Vert
&=\frac{1}{2}\sup_{-1\leq F\leq 1}\left| E \left[F(\mathcal{M}_n^{\bullet}) \left(1- \frac{1}{\text{Card}\,V(\mathcal{M}_n^{\bullet})} \frac{1}{E(1/\text{Card}\, V(\mathcal{M}_n^{\bullet}))}  \right) \right] \right| \\
&\leq  E \left[ \left| 1-\frac{1}{\text{Card}\,V(\mathcal{M}_n^{\bullet})} \frac{1}{E(1/\text{Card}\, V(\mathcal{M}_n^{\bullet}))}  \right| \right] \\
&= E\left[\left| 1- \frac{1/X_n}{E(1/X_n)} \right| \right] \\
& \underset{n \to \infty} {\longrightarrow} 0. \\
\end{aligned}$$
This completes the proof of Proposition \ref{variationtotale}.

\end{proof}

\textbf{Acknowledgement}. I am deeply indebted to Jean-Fran\c cois Le
Gall for suggesting me to study this problem, for stimulating
discussions and for carefully reading the manuscript and making many
useful suggestions.
I am grateful to Jérémie Bettinelli for a very interesting discussion concerning Proposition \ref{variationtotale}. 
I also thank the referee of this article for his careful reading and several helpful suggestions.

\end{document}